\documentclass[final,times,authoryear,1p]{elsarticle}

\usepackage{amsmath,amsfonts,amssymb}
\usepackage{amsthm}
\usepackage{graphicx}
\usepackage{multirow}
\usepackage{color}
\usepackage{fancyhdr}
\newtheorem{theorem}{Theorem}[section]
\newtheorem{lemma}[theorem]{Lemma}

\newtheorem{corollary}[theorem]{Corollary}
\newtheorem{definition}[theorem]{Definition}
\newtheorem{remark}[theorem]{Remark}

\usepackage{amssymb}

\journal{Transportation Research Part B}

\begin{document}

\begin{frontmatter}



 \begin{center}
 \textcolor{blue}{ARTICLE LINK:  http://www.sciencedirect.com/science/article/pii/S0191261514000022
\\  PLEASE CITE THIS ARTICLE AS\\ 
Han, K., Gayah, V.V., Piccoli, B., Friesz, T.L., Yao, T., 2015. On the continuum approximation of the on-and-off signal control on dynamic traffic networks. \\Transportation Research Part B 61, 73-97.}
 \line(1,0){469}
 \end{center}

\title{On the continuum approximation of the on-and-off signal control on dynamic traffic networks}


\author[ic]{Ke Han \corref{cor}}
\ead{k.han@imperial.ac.uk}

\author[cee]{Vikash V. Gayah}
\ead{gayah@engr.psu.edu}

\author[math]{Benedetto Piccoli}
\ead{piccoli@camden.rutgers.edu}

\author[ie]{Terry L. Friesz} 
\ead{tfriesz@psu.edu}

\author[ie]{Tao Yao}
\ead{tyy1@engr.psu.edu}

\cortext[cor]{Corresponding author}

\address[ic]{Department of Civil and Environmental Engineering, Imperial College London, United Kingdom.}
\address[cee]{Department of Civil and Environmental Engineering, Pennsylvania State University, USA.}
\address[math]{Department of Mathematical Sciences
and CCIB, Rutgers University - Camden, USA}
\address[ie]{Department of Industrial and Manufacturing Engineering, Pennsylvania State University, USA.}

\begin{abstract}
In the modeling of traffic networks, a signalized junction is typically treated using a binary variable to model the on-and-off nature of signal operation. While accurate, the use of binary variables can cause problems when studying large networks with many intersections. Instead, the signal control can be approximated through a continuum approach where the on-and-off control variable is replaced by a priority parameter. Advantages of such approximation include elimination of the need for binary variables, lower time resolution requirements, and more flexibility and robustness in a decision environment. It also resolves the issue of discontinuous travel time functions arising from the context of dynamic traffic assignment.  

Despite these advantages in application, it is not clear from a theoretical point of view how accurate is such continuum approach; i.e., to what extent is this a valid approximation for the on-and-off case.  The goal of this paper is to answer these basic research questions and provide further guidance for the application of such continuum signal model.  In particular, by employing the {\it Lighthill-Whitham-Richards} model \citep{Lighthill and Whitham, Richards} on a traffic network, we investigate the convergence of the on-and-off signal model to the continuum model in regimes of diminishing signal cycles. We also provide numerical analyses on the continuum approximation error when the signal cycles are not infinitesimal.  As we explain, such convergence results and error estimates depend on the type of fundamental diagram assumed and whether or not vehicle spillback occurs in a network. Finally,  a traffic signal optimization problem is presented and solved which illustrates the unique advantages of applying the continuum signal model instead of the on-and-off one.
\end{abstract}


\end{frontmatter}

\section{\label{Intro}Introduction}

Signalized intersections play a vital role in the design, management and control of urban traffic networks. These locations are often the most restrictive bottlenecks, and therefore urban traffic control strategies tend to focus on the operation of signalized intersections \citep{Miller1963, Robertson1974, Shelby2004, CP, Guler2012, Gayah2012}. Thus, it is imperative that we are able to accurately predict traffic dynamics at these locations, and the resulting impact on a network. Fortunately, modeling these common junctions is relatively straightforward: for a given movement at an intersection, the impact of the signal on traffic dynamics is incorporated using a single binary variable. When the signal is green for the subject movement, the binary variable allocates the entirety of the downstream link's capacity to the downstream end of the subject approach. When the signal is red, the binary variable ensures that this capacity is zero. 

Unfortunately, the discrete nature of this `on-and-off' signal timing makes studying and optimizing the control parameters of these junctions rather complex, especially on large networks with many signalized intersections. Incorporating the binary traffic signal state variables in a signal optimization process usually results in mixed integer mathematical programs; examples include \citet{Improta and Cantarella}, \cite{Lo 1999a} and \cite{Lo 1999b}. For large networks, these mixed integer mathematical programs can be very difficult to solve exactly. Even when possible, the solutions require a tremendous amount of time, which makes real-time applications impossible. Realistic extensions that account for the combination of {\it dynamic traffic assignment} (DTA) with signal optimization, such as the so-called {\it dynamic user equilibrium with signal control} (DUESC) problems, become especially difficult \citep{Aziz and Ukkusuri, Ukkusuri et al}. 

To simplify the modeling of signalized intersections in networks, recent studies have proposed an elegant continuum model to approximate traffic dynamics at traffic signals in an intuitive way \citep{Smith, GZ}. This model works as follows. Consider a simple merge junction with two incoming links, $I_1$ and $I_2$, and one outgoing link, $I_3$, as depicted in Figure~\ref{figsimplemerge}. Assume now that the junction $A$ is controlled by a fixed-cycle traffic signal which controls the movement of the exit flows on links $I_1$ and $I_2$. The receiving capacity of the outgoing link $I_3$ is assumed to be time-dependent and given by the supply function $S_3(t)$. Additionally, the fraction of the cycle dedicated to link $I_1$ is given by $\eta$, and the fraction dedicated to link $I_2$ is $1-\eta$ for some $\eta\in(0,\,1)$. The continuum model asserts that the proportions $\eta$ and $1-\eta$ of the downstream link capacity $S_3(t)$ are assigned to links $I_1$ and $I_2$, respectively, during the entirety of the signal cycle. A more detailed and formal definition of such model will be provided later in Section \ref{secSignal}.

While this continuum model will not predict traffic dynamics at the intersection exactly, it does have a number of advantages when compared to the on-and-off signal model that is typically used: 
\begin{itemize}
\item  In a discrete-time setting, the binary representation of signal control strategies will be replaced with a real-valued parameter $\eta$. This eliminates the need of using binary variables for the signalization, and significantly reduces the computational burden of the mixed integer programs such as those reviewed above. 

 \item  The on-and-off signal model usually demands a very fine time resolution to accommodate certain signal splits. For example, a cycle with 35 seconds of green phase and 25 seconds of red phase requires a time step of at most 5 seconds to be properly implemented. These fine time resolutions increase the computational requirements of the network simulations and/or optimizations. On the other hand such constraints do not apply to the continuum case, thus one has more flexibility in choosing the time step for computational convenience and efficiency.  

\item For any fixed-cycle traffic signal optimization problem defined on a prescribed time grid, the on-and-off signal strategy can only take on several discrete values, while the continuum model yields a continuous spectrum of choices and outcomes.

\item  The on-and-off signal control naturally results in discontinuities in travel time functions, which poses difficulties in quite  a few dynamic traffic assignment models. For example, a dynamic user equilibrium problem \citep{LWRDUE} cannot be properly defined with the on-and-off signal controls unless some sort of indifference of drivers in travel time is introduced \citep{SL2006, GZ, Handissertation}. Such obstacle can be easily avoided by the continuum signal model.

\end{itemize}

Despite the appealing features of the continuum signal model mentioned above, the model has never been rigorously analyzed in connection with its counterpart, the on-and-off model. From an application point of view, it is of fundamental importance to identify circumstances where such continuum approximation accurately describes the aggregate behavior that exists at signalized intersections and, perhaps more importantly, to identify when it is invalid and may induce significant error. It is also worthwhile to investigate  to what extent the continuum model is a good approximation of the on-and-off signal control. Solving these objectives can help to identify situations where this approximation can be used, and when its advantages can be realized without sacrificing model accuracy. This serves as the motivation of the current paper and are fully addressed by the findings made herein.

Our analysis of the continuum signal model is based on the network extension of the Lighthill-Whitham-Richards conservation law model that explicitly captures the temporal and spatial distributions of congestion as well as vehicle spillback. This paper employs the most general assumptions on the fundamental diagram to ensure the validity of the convergence result and the error estimates. Our specific findings and/or contributions regarding the continuum signal model are summarized as follows.

\begin{itemize}

\item[1(a).] For a signalized network, if no spillback\footnote{In this paper, spillback refers to the situation where the entrance of a link is in the congested phase, causing its supply to drop and affecting the outcome of the junction as well as its immediately preceding links.} occurs on any junction, then the traffic evolution on the network using the on-and-off signal model converges to the solution using a corresponding continuum signal model as the signal cycles tend to zero.  This is true for any type of fundamental diagram assumed.

\item[1(b).] In application, when the signal cycles are not infinitesimal, the difference between the aforementioned two solutions are uniformly bounded if spillback does not occur anywhere in the network. This is again true for any type of fundamental diagram assumed. In addition,  expression of such uniform bound is provided explicitly.

\item[2(a).] If spillback occurs at some junction and lasts for a significant period (e.g., on the scale of several signal cycles), then the above convergence does not hold if the fundamental diagram is triangular, while the convergence continues to hold if the fundamental diagram has a strictly concave congested branch.

\item[2(b).] Under the same assumption of 2(a), and that the signal cycles are not infinitesimal, the difference between the two solutions grows with time, regardless of the fundamental diagram assumed. However, when using a fundamental diagram with a strictly concave congested branch, the difference is significantly smaller than when a triangular fundamental diagram is assumed. Again, the differences are provided explicitly for any type of fundamental diagram.

\item[3.] If spillback takes place and recurs on a smaller time scale (e.g., smaller than a signal cycle), which we call {\it transient spillback}, then the convergence of the solution with the on-and-off signal model to the one with the continuum signal model does not hold, when the signal cycles tend to zero. Moreover, the approximation error may be very large and grows with time when the signal cycles are not infinitesimal. These statements are true for any fundamental diagram.

\item[4.] We provide a traffic signal optimization procedure in the form of a mixed integer linear program which employs either the continuum traffic signal model or the on-and-off signal model. Performances and results of these two programs are compared which highlights the unique advantages of the continuum model and illustrates some of its solution characteristics in line with our theoretical results. 
\end{itemize}

The link dynamics employed in this paper are described by the following first order scalar conservation law:
\begin{equation}\label{LWRPDE}
\partial_t\,\rho(t,\,x)+\partial_x\, f\big(\rho(t,\,x)\big)~=~0 \qquad (t,\,x)\in [0,\,T]\times[a,\,b]
\end{equation}
where $[0,\,T]$ is some fixed time horizon and the link is expressed as a spatial interval $[a,\,b]$, $\rho(t,\,x)\in[0,\,\rho_j]$ denotes local vehicle density, where $\rho_{j}$ denotes the jam density. Regarding the fundamental diagram $f(\cdot):~[0,\,\rho_{j}]\rightarrow [0,\,C]$ where $C$ denotes the flow capacity, we impose the following very mild assumptions.\\

\noindent {\bf (F)} The fundamental diagram $f(\rho)$ is continuous and concave, and vanishes at $\rho=0$ and $\rho=\rho_{j}$. \\

\noindent Note that more restrictive constraints are sometimes used throughout this paper when considering the impacts of the continuum approximation when different functional forms of the fundamental diagram are considered.

Analysis of the conservation law models involves shocks and rarefaction waves, which are very case-sensitive especially in the presence of a family of signal controls. Therefore, attacking the proposed problem directly using conservation laws is quite difficult. As part of our contribution in methodology, we invoke the analytical framework of variational theory \citep{CC1, VT1, Newella} and weak value conditions \citep{ABP} to analyze the models of interest. As we shall demonstrate,  solution representation of the signalized network, asymptotic behavior of the solutions in regimes of diminishing signal cycles, as well as error estimates for the two types of signal models, are all tremendously simplified by considering the variational theory and the weak value conditions.

The rest of this paper is organized as follows. Section \ref{secSignal} introduces additional concepts and notations pertaining to signalized junctions, where the on-and-off and the continuum signal models are formally defined. Section \ref{secHJweak} reviews some essential background on variational theory and weak conditions. Section \ref{secwithoutspillback} establishes, in the absence of vehicle spillback, convergence of the on-and-off model to the continuum model as the signal cycles tend to zero. We also provide error estimates for the continuum approximation of the on-and-off model when the signal cycles are not infinitesimal. These results are independent of the type of fundamental diagram employed. In Section \ref{secwithspillback}, we conduct similar investigations of convergence and error estimates, assuming that spillback occurs and is sustained at a signalized intersection. The corresponding results are dependent on the type of fundamental diagram employed. Section \ref{secTransient} provides a discussion of convergence in the presence of transient and recurring spillback. Section \ref{secNumerical} supports the theoretical results from previous sections using numerical examples. In Section \ref{secApp}, we propose an application of the continuum signal model, namely, a mixed integer linear programming approach for optimal signal timing problem. This is used to illustrate the  modeling and computational advantages of the continuum signal model over the on-and-off one, as well as its solution characteristics and qualities. Finally, Section \ref{secConclusions} provides some concluding remarks.

\section{Signalized junction with fixed cycle and split}\label{secSignal}

In order to illustrate the key features of signalized junctions, we focus on a signalized merge node $A$, depicted in Figure \ref{figsimplemerge}. At this node, there are two incoming links, $I_1$ and $I_2$, and one outgoing link, $I_3$. We also note that an additional signal exists at the downstream node $B$ of link $I_3$, which will be used to discuss some queue spillbacks in Sections \ref{secwithspillback} and \ref{secNumerical}. Although all subsequent results are stated for such junction, extensions of our methodology and insights to more general junctions are straightforward, see more explanation in Section \ref{secConclusions}. 

In signal timing, a period containing one complete green phase and one complete red phase is called a cycle. The cycle length for $I_1$ (and also for $I_2$) is  denoted by $\Delta_A\in\mathbb{R}_{+}$.  Fix a split parameter $\eta_1\in(0,\,1)$. We let the green time for $I_1$ be $\eta_1 \Delta_A$ and the green time for $I_2$ be $(1-\eta_1)\Delta_A$ in a full cycle.

\begin{figure}[h!]
\centering
\includegraphics[width=.5\textwidth]{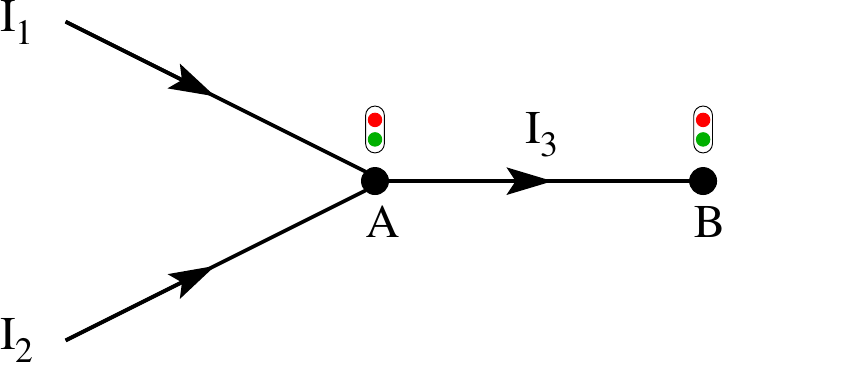}
\caption{A signalized merge junction.}
\label{figsimplemerge}
\end{figure}

\noindent Each link $I_i$  is expressed as a spatial interval $[a_i,\,b_i],\,i=1,2,3$. The density on each link is   $\rho_i(t,\,x),\,(t,\,x)\in[0,\,T]\times [a_i,\,b_i],\,i=1,2,3$. We define the demand functions $D_i(t)$ for $I_i$, $i=1,\,2$ and the supply function $S_3(t)$ for $I_3$  \citep{LK1999} as
\begin{align}
\label{demanddef}
D_i(t)~=~&\begin{cases}
C_i \qquad& \hbox{if}~~x=b_i-~~\hbox{is in the congested phase}
\\
f_i\big(\rho_i(t,\,b_i-)\big)\qquad & \hbox{if}~~x=b_i-~~\hbox{is in the uncongested  phase}
\end{cases}\qquad i=1,\,2
\\
\label{supplydef}
S_3(t)~=~&\begin{cases}
C_3 \qquad& \hbox{if}~~x=a_3+~~\hbox{is in the uncongested phase}
\\
f_3\big(\rho_3(t,\,a_3+)\big)\qquad & \hbox{if}~~x=a_3+~~\hbox{is in the congested  phase}
\end{cases}
\end{align}
where $f_i(\cdot)$ is the fundamental diagram, $C_i$ denotes the flow capacity, $i=1,\,2,\,3$. Based on $S_3(t)$, we define the {\it effective supplies} $\mathcal{S}_3^1(t)$ and $\mathcal{S}_{3}^2(t)$ associated with link $I_1$ and $I_2$ respectively as 
\begin{equation}\label{effsupply}
\mathcal{S}_3^1(t)~\doteq~\min\big\{C_1,\,S_3(t)\big\},\qquad  \mathcal{S}_3^2(t)~\doteq~\min\big\{C_2,\,S_3(t)\big\}
\end{equation}
Quantities $\mathcal{S}_3^1(t)$ and $\mathcal{S}_3^2(t)$ represent, respectively, the capacity provided by the downstream link that is available for $I_1$ and $I_2$ to utilize. They will be used to define the on-and-off and the continuum signal models as below.

We consider the periodic,  piecewise constant control functions $u_1(t)$ and $u_2(t):[0,\,T]\to \{0,\,1\}$, such that 
\begin{equation}\label{control1}
u_1(t)~=~\begin{cases} 1 \qquad\qquad &\hbox{when the signal is green for } I_1
\\
0 \qquad\qquad &\hbox{when the signal is red for } I_1
\end{cases}
\end{equation}
\begin{equation}\label{control2}
u_2(t)~=~\begin{cases} 1 \qquad\qquad &\hbox{when the signal is green for } I_2
\\
0 \qquad\qquad &\hbox{when the signal is red for } I_2
\end{cases}
\end{equation}
One obvious identity that must be satisfied by these controls is $u_1(t)+u_2(t)\equiv 1,\,\forall t\in[0,\,T]$.  One can now write the boundary flows corresponding to an on-and-off control as
\begin{equation}\label{oaodef}
\left.\begin{array}{l}
 f_1\big(\rho_1(t,\,b_1)\big)~=~\min\big\{D_1(t),~~ \mathcal{S}^1_3(t)\cdot u_1(t)\big\}
\\
f_2\big(\rho_2(t,\,b_2)\big)~=~\min\big\{D_2(t),~~ \mathcal{S}^2_3(t)\cdot u_2(t)\big\}
\\
 f_3\big(\rho_3(t,\,a_3)\big)~=~ f_1\big(\rho_1(t,\,b_1)\big)+f_2\big(\rho_2(t,\,b_2)\big)
\end{array}
\right\}\hbox{On-and-Off Signal Model}
\end{equation}
where $f_1\big(\rho_1(t,\,b_1)\big)$, $f_2\big(\rho_2(t,\,b_2)\big)$ denote the exit flows of $I_1$ and $I_2$; $ f_3\big(\rho_3(t,\,a_3)\big)$ denotes the inflow of $I_3$.  On the other hand, the continuum signal model states that
\begin{equation}\label{continuumdef}
\left.\begin{array}{l}
f_1\big(\rho_1(t,\,b_1)\big)~=~\min\big\{D_1(t),~~ \eta_1\mathcal{S}_3^1(t)\big\}
\\
f_2\big(\rho_2(t,\,b_2)\big)~=~\min\big\{D_2(t),~~ (1-\eta_1)\mathcal{S}_3^2(t)\big\}
\\
 f_3\big(\rho_3(t,\,a_3)\big)~=~ f_1\big(\rho_1(t,\,b_1)\big)+f_2\big(\rho_2(t,\,b_2)\big)
\end{array}
\right\}\hbox{Continuum Signal Model}
\end{equation}

\begin{remark}
It should be noted that we assume here, as in nearly all first-order traffic flow models, that vehicles accelerate and decelerate instantaneously. Of course, acceleration rates are bounded in reality, and this complication will introduce an additional source of error. In practice, this is usually accounted for by including lost times at the signal where flow is zero, and modeling saturation flows during the effective green time. The methodological framework presented in this paper can be easily modified to incorporate the inclusion of lost times and/or yellow times: simply relax the assumption that the sum of the priority parameters is equal to one and instead let this sum be equal to the fraction of the cycle during which vehicles are allowed to discharge at saturation. This fraction can usually be detrained fairly easily in practice for a given cycle length.
\end{remark}

Although subsequent analyses regarding convergence and error estimate are only stated for link $I_1$,  the treatment of $I_2$ is completely symmetric and quite similar.

\section{The Hamilton-Jacobi equation and the weak boundary conditions}\label{secHJweak}

We introduce the Moskowitz function $N(t,\,x)$ \citep{Moskowitz} , also know as the Newell-curve \citep{Newella}, which measures the cumulative number of vehicles that have passed location $x$ by time $t$. The function $N(t,\,x)$ satisfies the following Hamilton-Jacobi equation
\begin{equation}\label{HJE}
\partial_t N(t,\,x)-f\big(- \partial_x N(t,\,x)\big)~=~0\qquad (t,\,x)\in[0,\,T]\times[a,\,b]
\end{equation}
subject to initial condition, upstream and downstream boundary conditions, to be defined below.

\subsection{The generalized Lax-Hopf formula}
Our analysis of the signalized junction involves a semi-analytical solution representation of the Hamilton-Jacobi equation \eqref{HJE} known as the {\it generalized Lax-Hopf formula} \citep{ABP, CC1, CC2}. For the convenience of invoking weak conditions, we employ a class of lower-semicontinuous viability episolutions of \eqref{HJE} in the sense of Barron-Jensen/Frankowska \citep{BJ, Frankowska}.

Let us fix a temporal-spatial domain $[0,\,T]\times [a,\,b]$ where $[0,\,T]$ is the time horizon, $b-a=L$ is the length of the link. The articulation of the generalized Lax-Hopf formula requires the following definition of value conditions.

\begin{definition}
A value condition $\mathcal{C}(\cdot,\,\cdot)$ is a lower-semicontinuous function that maps $\Omega$, a subset of $[0,\,T]\times[a,\,b]$, to $\mathbb{R}$. 
\end{definition}

\begin{theorem}{\bf (Generalized Lax-Hopf formula)}
The viability episolution to \eqref{HJE} associated with value condition $\mathcal{C}(\cdot,\,\cdot)$ is given by 
\begin{equation}\label{gLH}
N_{\mathcal{C}}(t,\,x)~=~\inf_{(u,\,\tau)\in Dom(f^*)\times \mathbb{R}_+}\big\{\mathcal{C}(t-\tau,\,x-\tau u)+\tau f^*(u)\big\}
\end{equation}
where $f^*(\cdot)$ is the concave transformation of the Hamiltonian $f(\cdot)$:
$$
f^*(u)~=~\sup_{\rho\in [0,\,\rho_{j}]}\big\{f(\rho)-u\rho\big\}
$$
and $Dom(f^*)\doteq [f'(0+),\,f'(\rho_j-)]$.
\end{theorem}
\begin{proof}
The reader is referred to \cite{ABP} for a proof. 
\end{proof}

\subsection{The weak value conditions}

Let us introduce, for equation \eqref{HJE}, the initial condition $N_{ini}(x)$, the upstream boundary condition $N_{up}(t)$, and the downstream boundary condition $N_{down}(t)$. We make note of the fact that a viability episolution $N(t,\,x)$ given by \eqref{gLH} needs only satisfy the above conditions in an inequality $(\leq)$ sense. In other words, we have
\begin{equation}\label{inqconstraints1}
 N(0,\,x)~\leq~N_{ini}(x)\qquad\forall x\in[a,\,b]
\end{equation}
\begin{equation}\label{inqconstraints2}
N(t,\,a)~\leq~N_{up}(t),\qquad N(t,\,b)~\leq~N_{down}(t)\qquad \forall t \in[0,\,T]
\end{equation}
The reader is referred to \cite{ABP} and \cite{CC1} for more detailed explanation. 

\begin{remark}
Throughout this paper, $N_{up}(t)$ is taken as the time-integral of the demand function of the preceding link, rather than the time-integral of the link entry flow; similarly, $N_{down}(t)$ is taken as the time-integral of the supply function of the following link, rather than the time-integral of the link exit flow. Since the demand (supply) of preceding (following) link is usually larger than the actual inflow (outflow) of the link of interest, the boundary conditions $N_{up}(t)$ and $N_{down}(t)$ are satisfied only in an inequality ($\leq$) sense; in other words, they are weak boundary conditions \citep{ABP}. The advantage of invoking those weak boundary conditions is that solution representation (Lax-Hopf formula) on the link of interest can be obtained without knowledge of the actual inflow or outflow. Moreover, in order to draw any conclusion about the solution, it suffices to investigate $N_{up}(t)$ and $N_{down}(t)$. These two facts greatly simplify our analyses that will follow in Section \ref{secwithoutspillback} and \ref{secwithspillback}.
\end{remark}

In the presence of weak conditions $N_{ini}(x),\,N_{up}(t)$ and $N_{down}(t)$, the Lax-Hopf formula \eqref{gLH} can be instantiated as follows.

\begin{theorem}\label{thmLHib}{\bf(Lax-Hopf formula with weak initial and boundary conditions)}
Consider the Hamilton Jacobi equation \eqref{HJE} with a continuous, concave Hamiltonian $f(\cdot)$, and define $v\doteq f'(0+)$, $-w\doteq f'(\rho_j-)$. Given weak conditions $N_{ini}(x),\,N_{up}(t)$ and $N_{down}(t)$, the viability episolution given by \eqref{gLH} can be explicitly expressed as
\begin{align}
\label{explicitLH1}
N(t,\,x)~=~&\min_{u\in[-w,\,v]}A(u; t, x)\qquad (t,\,x)\in\Omega_I
\\
\label{explicitLH2}
N(t,\,x)~=~&\min\left\{\min_{u\in[{x-a\over t},\,v]}B(u; t, x),~~ \min_{u\in[-w,\,{x-a\over t}]}A(u; t, x)\right\} \qquad (t,\,x)\in\Omega_{II}
\\
\label{explicitLH3}
N(t,\,x)~=~&\min\left\{\min_{u\in[{x-b\over t},\,v]}A(u; t, x),~~ \min_{u\in[-w,\,{x-b\over t}]}C(u; t, x)\right\} \qquad (t,\,x)\in\Omega_{III}
\\
\label{explicitLH4}
N(t,\,x)~=~&\min\left\{\min_{u\in[{x-a\over t},\,v]}B(u; t, x),~ \min_{u\in[-w,\,{x-b\over t}]}C(u; t, x),~ \min_{u\in[{x-b\over t},\,{x-a\over t}]}A(u; t, x)\right\} \quad (t,\,x)\in\Omega_{IV}
\end{align}
where 
\begin{align}
A(u; t, x)~=~&N_{ini}(x-ut)+tf^*(u) 
\\
B(u; t, x)~=~&N_{up}\left(t-{x-a\over u}\right)+{x-a\over u} f^*(u)
\\
C(u; t, x)~=~&N_{down}\left(t-{x-b\over u}\right)+{x-b\over u}f^*(u)
\end{align}
and
\begin{equation}\label{Omegadef}
\begin{array}{r}
\Omega_I~=~\left\{(t,\,x)\in[0,\,T]\times[a,\,b]:~  x~\geq~a+vt ,~ x~\leq~b-wt \right\}
\\
\Omega_{II}~=~\left\{ (t,\,x)\in(0,\,T]\times[a,\,b]:~ x~<~a+vt,~ x~\leq~b-wt   \right\}
\\
\Omega_{III}~=~\left\{ (t,\,x)\in(0,\,T]\times[a,\,b]:~x~\geq~a+vt,~ x~>~b-wt   \right\}
\\
\Omega_{IV}~=~\left\{ (t,\,x)\in(0,\,T]\times[a,\,b]:~  x~<~a+vt,~ x~>~b-wt\right\}
\end{array}
\end{equation}
\end{theorem}
\begin{proof}
The speeds of the kinematic waves fall within the interval $[-w,\,v]$, thus the temporal-spatial domain $[0,\,T]\times[a,\,b]$ can be partitioned into four parts, depending on whether or not a point $(t,\,x)$ can be influenced by the initial, upstream boundary, and downstream boundary conditions; see Figure \ref{figdomains1} for an illustration. To establish \eqref{explicitLH1}-\eqref{explicitLH4}, it suffices, for each subregion, to locate the domain of influence and apply the Lax-Hopf formula \eqref{gLH}. 
\end{proof}

\begin{figure}[h!]
\centering
\includegraphics[width=.5\textwidth]{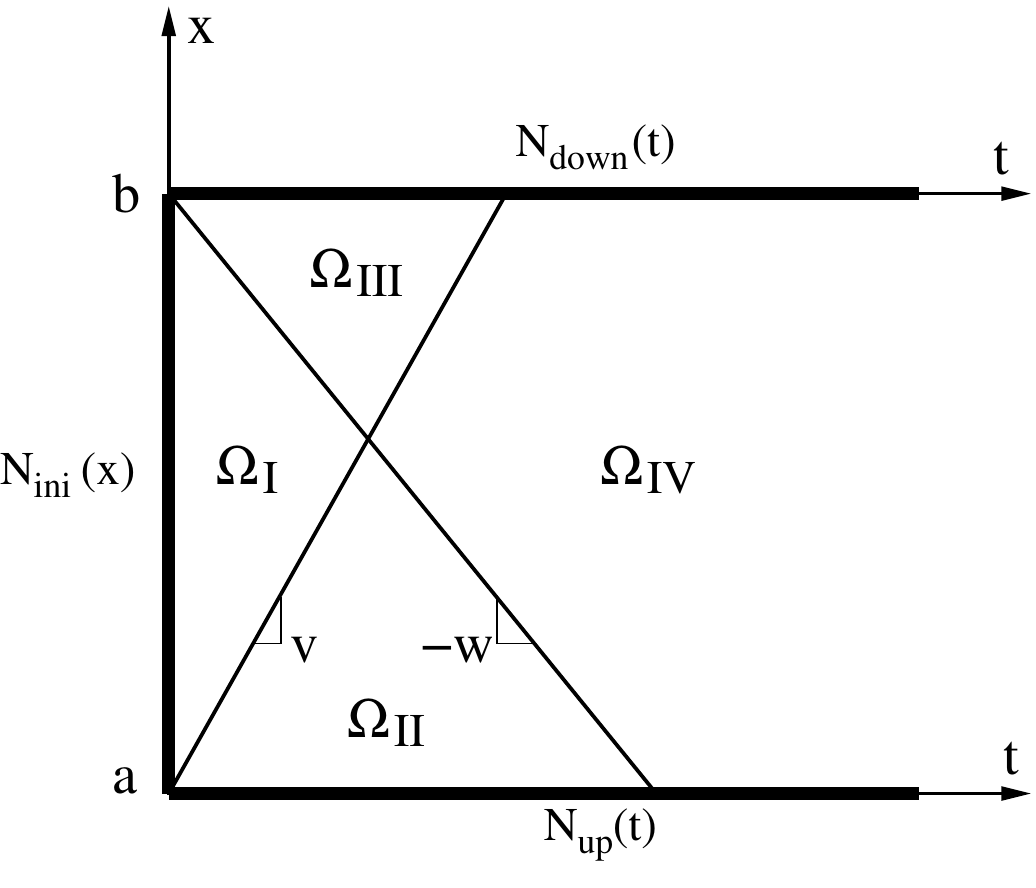}
\caption{Partition of the temporal-spatial domain into four subregions. In $\Omega_I$, the solution is influenced by $N_{ini}(\cdot)$ only; in $\Omega_{II}$, the solution is influenced by $N_{up}(\cdot)$ and $N_{ini}(\cdot)$ only; in $\Omega_{III}$, the solution is influenced by $N_{down}(\cdot)$ and $N_{ini}(\cdot)$; and in $\Omega_{IV}$, the solution is influenced by $N_{ini}(\cdot)$, $N_{up}(\cdot)$ and $N_{down}(\cdot)$.}
\label{figdomains1}
\end{figure}

In view of the signalized junction in Figure \ref{figsimplemerge} and the two signal models expressed in \eqref{oaodef} and \eqref{continuumdef}, we define the weak downstream boundary conditions for the H-J equation on link $I_1$ respectively as
\begin{equation}\label{NDeltadef}
N_{down}^{\Delta_A}(t)~\doteq~\int_0^t \mathcal{S}_3^1(\tau)\,u_1(\tau)\,d\tau \qquad\hbox{(the on-and-off case)}
\end{equation}
and 
\begin{equation}\label{N0def}
N_{down}^0(t)~\doteq~\int_0^t \eta_1 \mathcal{S}_3^1(\tau)\,d\tau  \qquad\hbox{(the continuum case)}
\end{equation}

\section{Convergence results and error estimation in the absence of spillback}\label{secwithoutspillback}

In this section, under the assumption that no spillback occurs at the signalized junction, we are interested in finding out the asymptotic behavior of the on-and-off signal model, and whether or not it converges to the continuum model, when the signal cycle length tends to zero. As we shall explain, the no-spillback assumption is crucial for the established results below. The case with spillback will be presented later in Section \ref{secwithspillback}. Note that all the results presented in the rest of this section are valid for any type of fundamental diagram as long as the very mild assumptions {\bf (F)} mentioned in the introduction are satisfied.

\subsection{Convergence result without spillback}\label{secconvns}
Let us re-visit the merge junction $A$ depicted in Figure \ref{figsimplemerge}. The absence of spillback at node A implies that the entrance of link $I_3$ remains in the uncongested phase. In other words, the supply function $S_3(t)$ of $I_3$ is equal to its flow capacity. The following convergence theorem holds under such circumstance.

\begin{theorem}\label{thmconvns}
Consider the merge junction depicted in Figure \ref{figsimplemerge}, and a signal control $u_1(t)$ for link $I_1$ with cycle  $\Delta_A$ and  split parameter $\eta_1\in(0,\,1)$.  We let $N_{up}(t)$ and $N_{ini}(x)$ be the weak upstream boundary condition and initial condition for the H-J equation \eqref{HJE} on $I_1$. Let $N^{\Delta_A}(t,\,x)$ and $N^0(t,\,x)$ be the solutions of the H-J equation with additional downstream boundary conditions $N_{down}^{\Delta_A}(t)$ and $N^0_{down}(t)$ respectively, where $N^{\Delta_A}_{down}(t)$ and $N^0_{down}(t)$ are given by \eqref{NDeltadef} and \eqref{N0def}.   Furthermore, assume that the entrance of link $I_3$ remains in the uncongested phase. Then $N^{\Delta_A}(t,\,x) \to N^0(t,\,x)$ uniformly for all $(t,\,x)\in[0,\,T]\times[a_1,\,b_1]$, as $\Delta_A\rightarrow 0$.
\end{theorem}
\begin{proof}
We readily notice that the control function $u_1(t)$ defined in \eqref{control1}  converges weakly to the constant function $\eta_1$ on the time interval $[0,\,T]$ as $\Delta_A\to 0$. According to the no-spillback hypothesis, the effective supply $\mathcal{S}_3^1(t) \doteq \min\{C_1,\,S_3(t)\}\equiv \min\{C_1,\,C_3\}$ where $C_3$ denotes the flow capacity of $I_3$. Therefore  by the definition of weak convergence, we have
\begin{equation}\label{weakconveqn2}
N_{down}^{\Delta_A}(t)~=~\int_0^t \mathcal{S}_3^1(\tau)u_1(\tau)\,d\tau~~\longrightarrow~~ \int_0^t\eta_1\mathcal{S}_3^1(\tau)\,d\tau~=~N_{down}^0(t)\qquad \hbox{as}~~\Delta_A~\to~0
\end{equation}
uniformly for all $t\in[0,\,T]$.

By \eqref{explicitLH1} and \eqref{explicitLH2}, we deduce that $N^{\Delta_A}(t,\,x)\equiv N^0(t,\,x)$ for all $(t,\,x)\in \Omega_{I}\cup\Omega_{II}$, since the solution in these regions is not affected by the downstream boundary condition.

We next turn our attention to region $\Omega_{III}\cup\Omega_{IV}$. Given any $\varepsilon>0$, by virtue of \eqref{weakconveqn2}, there exists a $\delta>0$ such that whenever $\Delta_A<\delta$, we have 
\begin{equation}\label{epsilon}
\left|N_{down}^{\Delta_A}(t)-N_{down}^0(t)\right|~<~\varepsilon\qquad \forall t\in[0,\,T]
\end{equation}
Fix arbitrary $(t,\,x)\in\Omega_{III}\cup\Omega_{IV}$, we denote
\begin{align}
\label{minone}
u^{*,\Delta_A}~\doteq~&\text{argmin}_{u\in[-w,\,{x-b\over t}]}\left\{N^{\Delta_A}_{down}\left(t-{x-b\over u}\right)+{x-b\over u}f^*(u)   \right\}
\\
\label{mintwo}
u^{*,0}~\doteq~&\text{argmin}_{u\in[-w,\,{x-b\over t}]}\left\{N^{0}_{down}\left(t-{x-b\over u}\right)+{x-b\over u}f^*(u)   \right\}
\end{align}
According to \eqref{epsilon} and \eqref{mintwo},
\begin{align*}
N^{\Delta_A}_{down}\left(t-{x-b\over u^{*,\Delta_A}}\right)+{x-b\over u^{*,\Delta_A}}f^*\big(u^{*,\Delta_A}\big)~\geq~& N^{0}_{down}\left(t-{x-b\over u^{*,\Delta_A}}\right)+{x-b\over u^{*,\Delta_A}}f^*\big(u^{*,\Delta_A}\big)-\varepsilon
\\
~\geq~& N^{0}_{down}\left(t-{x-b\over u^{*,0}}\right)+{x-b\over u^{*,0}}f^*\big(u^{*,0}\big)-\varepsilon
\end{align*}
We may similarly deduce from \eqref{epsilon} and \eqref{minone} that
$$
N^{0}_{down}\left(t-{x-b\over u^{*,0}}\right)+{x-b\over u^{*,0}}f^*\big(u^{*,0}\big)~\geq~N^{\Delta_A}_{down}\left(t-{x-b\over u^{*,\Delta_A}}\right)+{x-b\over u^{*,\Delta_A}}f^*\big(u^{*,\Delta_A}\big)-\varepsilon
$$
Thus
\begin{multline}
\Bigg|\min_{u\in[-w ,\, {x-b\over t}]}\left\{N^{\Delta_A}_{down}\left(t-{x-b\over u}\right)+{x-b\over u}f^*(u)   \right\} - 
\\
\min_{u\in[-w ,\, {x-b\over t}]}\left\{N^0_{down}\left(t-{x-b\over u}\right)+{x-b\over u}f^*(u)   \right\}\Bigg|~<~\varepsilon
\end{multline}
In view of \eqref{explicitLH3} and \eqref{explicitLH4}, the above estimate gives the difference in $\min_{u\in[-w, {x-b\over t}]}C(u; t, x)$ when $N^{\Delta_A}_{down}(\cdot)$ and $N^0_{down}(\cdot)$ are respectively used.  Since the rest of the quantities appearing in \eqref{explicitLH3}-\eqref{explicitLH4} do not depend on the downstream boundary condition,  we conclude that $|N^{\Delta_A}(t,\,x)-N^0(t,\,x)|<\varepsilon$ for all $(t,\,x)\in\Omega_{III}\cup\Omega_{IV}$. This implies the desired uniform convergence. 
\end{proof}

\begin{remark}
One important observation  from the proof of Theorem \ref{thmconvns}  is that when conditions $N_{ini}(x)$ and $N_{up}(t)$ are fixed, the difference of the two Moskowitz functions $\left|N^{\Delta_A}(t,\,x)-N^0(t,\,x)\right|$ is bounded by the maximum difference of their respective weak downstream boundary conditions, that is,
$$
\left|N^{\Delta_A}(t,\,x)-N^0(t,\,x)\right|~\leq~\max_{\tau\in[0,\,t]}\left|N_{down}^{\Delta_A}(\tau)-N_{down}^0(\tau)\right|\qquad\forall t\in[0,\,T],\quad x\in[a,\,b]
$$
 In other words, approximation of the Moskowitz function is only ``as good as" how $N_{down}^{\Delta_A}(t)$ is approximated by $N_{down}^0(t)$. Such insight is crucial for our error analysis presented later.
\end{remark}

The next corollary generalizes the convergence result stated for a single junction to a network.

\begin{corollary}
Consider a  network with a fixed-cycle-and-split signal control at each intersection, with merge rules given by \eqref{oaodef}. Assume the flow dynamic on each link is governed by a scalar conservation law \eqref{LWRPDE} with a continuous and concave fundamental diagram. In addition, assume that the entrance of every link remains in the uncongested phase; i.e., no spillback occurs in the network. Then the solution on this network converges to the one corresponding to the continuum signal model \eqref{continuumdef}, when the traffic signal cycles tend to zero.
\end{corollary}
\begin{proof}
Without loss of generality we assume that the network is initially empty. Under the stated hypothesis, for each link, the supply function of this link is always a constant and equal to its flow capacity.  According to Theorem \ref{thmconvns}, convergence to a continuum model holds at each signalized intersection. Further notice that the error of the continuum approximation \eqref{continuumdef} on each link adds up linearly throughout the network. Thus the convergence also holds on a network level.
\end{proof}

\subsection{Approximation errors without spillback}

In the previous section, we have established convergence results for the continuum signal models in regimes where the signal cycles are infinitesimal. While providing theoretical foundations for a class of approximate signalized junction models, these convergence results are not satisfying from a practical point of view since any signal cycle must be bounded away from zero. To make our investigations more practical, we conduct further analysis on the approximation error of the continuum model when the cycles are not infinitesimal. Results presented below may assist practitioners with applying the continuum approximation of the on-and-off signal models and evaluating its efficacy.

\begin{theorem}\label{estthm1}{\bf (Error estimate without spillback)}
Consider the signalized merge junction depicted in Figure \ref{figsimplemerge}. Assume the Hamilton-Jacobi equation \eqref{HJE} for link $I_1$ has weak value conditions $N_{ini}(x)$ and $N_{up}(t)$. Furthermore, let $N^{\Delta_A}(t,\,x)$ and $N^0(t,\,x)$ be the solutions of this H-J equation with additional downstream boundary conditions $N_{down}^{\Delta_A}(t)$ and $N_{down}^0(t)$ respectively, where $N_{down}^{\Delta_A}(t)$ and $N_{down}^0(t)$ are given in \eqref{NDeltadef} and \eqref{N0def}. In addition, assume that the entrance of link $I_3$ remains in the uncongested phase. Then for all $(t,\,x)\in[0,\,T]\times[a_1,\,b_1]$, 
\begin{equation}\label{altproofconv1}
\left|N^{\Delta_A}(t,\,x)-N^0(t,\,x)\right|~\leq~\eta_1(1-\eta_1)\Delta_A \min\{C_1, C_3\}~\leq~{1\over 4} \Delta_A \min\{C_1,C_3\} 
\end{equation}
\end{theorem}
\begin{proof}
According to the hypothesis, we have that $\mathcal{S}_3^1(t)\equiv \min\{C_1,\,C_3\}$. Then for any $t\geq 0$,
\begin{equation}\label{nsestimate}
\left|N_{down}^{\Delta_A}(t) -N_{down}^0(t)\right|~=~\min\{C_1, C_3\}\left|\int_0^tu_1(\tau)\,d\tau- \eta_1 t \right|~\leq~\eta_1(1-\eta_1)\Delta_A \min\{C_1,C_3\}
\end{equation}
 Therefore, using the same argument as in the proof of Theorem \ref{thmconvns}, we conclude that 
 $$
 \left|N^{\Delta_A}(t,\,x) -N^0(t,\,x)\right|~\leq~\max_{\tau\in[0,\,t]}\left|N_{down}^{\Delta_A}(\tau) -N_{down}^0(\tau)\right| ~\leq~\eta_1(1-\eta_1)\Delta_A \min\{C_1,\,C_3\}
 $$ 
 for all $(t,\,x)\in[0,\,T]\times[a_1,\,b_1]$.  Finally, it is useful to notice that $\eta_1(1-\eta_1)\leq {1\over 4}$,  and the equality holds if and only if $\eta_1={1\over 2}$. 
\end{proof}

In parallel to Section \ref{secconvns}, we state the error estimates for a whole network in the corollary below.

\begin{corollary}\label{cor1}
Consider a  network with a fixed-cycle signal control at each intersection, with merge rules given by \eqref{oaodef}. Assume the flow dynamic on each link is governed by a scalar conservation law \eqref{LWRPDE} with a continuous and concave fundamental diagram. In addition, assume that the entrance of every link remains in the uncongested phase. Then for every link $I_i$ of the network, let $N^{\Delta, i}(t,\,x)$ and $N^{0, i}(t,\,x)$ be  the two Moskowitz functions obtained from the on-and-off and the continuum modeling approaches respectively. Then $|N^{\Delta,i}(t,\,x)-N^{0,i}(t,\,x)|$ is less than or equal to the sum of such differences of preceding links including $I_i$. 
\end{corollary}
\begin{proof}
Without loss of generality, we again assume that the network is initially empty. If the upstream condition $N_{up}^i(t)$ of $I_i$ is fixed, the error $|N^{\Delta,i}(t,\,x)-N^{0,i}(t,\,x)|$ is estimated by \eqref{altproofconv1}. However, $N_{up}^i(t)$ may be subject to errors due to the continuum signal approximations on the preceding links. Clearly such error adds up linearly throughout the network, therefore the actual  $|N^{\Delta,i}(t,\,x)-N^{0,i}(t,\,x)|$ is bounded by a summation of errors from previous links.
\end{proof}

\begin{remark}
The error estimation established in Corollary \ref{cor1} is quite conservative in that it assumes the worst case scenario where errors add up without cancellation. In order to obtain a more accurate estimation, one needs to look at specific network topology and the actual signal split and cycles used, including the initial phase of each signal. The key message from Corollary \ref{cor1} is that the continuum  approximation picks up errors from one link to another additively at the most.
\end{remark}

\section{Convergence results and error estimation in the presence of spillback}\label{secwithspillback}

In the previous section, asymptotic behavior of the on-and-off signal model as well as the approximation error of its continuum counterpart are investigated under the assumption that the entrance of $I_3$ remains in the uncongested phase. On the other hand, if $I_3$ is dominated by the congested phase,  the signal control $u_3(t)$ located at node $B$ may directly affect supply $S_3(t)$ and hence $\mathcal{S}_3^1(t)$. As we shall explain below, when this happens the previously stated convergence result and error estimation no longer hold in certain cases, depending on the type of fundamental diagram employed. 

For the remainder of this section, we will examine the behavior of the continuum approximation model assuming that spillback persists for a significant period of time (i.e., for at least several signal cycles). The case where spillback persists for a much shorter period of time will be discussed in Section \ref{secTransient}.

\subsection{Convergence result with spillback}

\subsubsection{The case with triangular fundamental diagram}

We will first demonstrate analytically that the convergence does not hold with vehicle spillback if a triangular fundamental diagram is assumed. A triangular fundamental diagram has the following form:
\begin{equation}\label{triangularfd}
f(\rho)~=~\begin{cases}
v\rho \qquad\qquad & \rho\in[0,\,\rho_c]
\\
-w(\rho-\rho_{j})\qquad\qquad & \rho\in(\rho_c,\,\rho_{j}]
\end{cases}
\end{equation}
where $v$ and $w$ are the speeds of the forward- and backward-propagating kinematic waves respectively. In the case of triangular fundamental diagram, $v$ also coincides with the free-flow speed. $\rho_c$ denotes the critical density, and $\rho_{j}$ denotes the jam density.

Consider the merge junction in Figure \ref{figsimplemerge}. Let us focus on $I_3$ which remains in the congested phase. In the spatial-temporal domain of $I_3$, characteristic lines with slope $-w_3$ emit from the right boundary $x=b_3$ and reach the left boundary $x=a_3$, where $w_3$ denotes the backward wave speed on link $I_3$ (see Figure \ref{figScenario4}). When the light is red, the exit flow $q_3$ is equal to zero, creating a kinematic wave with speed $-w_3$ and density value $\rho^3_j$ where $\rho^3_j$ denotes the jam density of $I_3$; when the light is  green, the exit flow $q_3$ is equal to the flow capacity $C_3$, creating a kinematic wave with speed $-w_3$ and density value $\rho^3_c$, where $\rho^3_c$ denotes the critical density on $I_3$. As a result,  the supply function $S_3(t)$ at the entrance of $I_3$ fluctuates between $0$ and $C_3$, leading $\mathcal{S}_3^1(t)$ to fluctuate between $0$ and $\min\{C_1,\,C_3\}$.

\begin{figure}[h!]
\centering
\includegraphics[width=.6\textwidth]{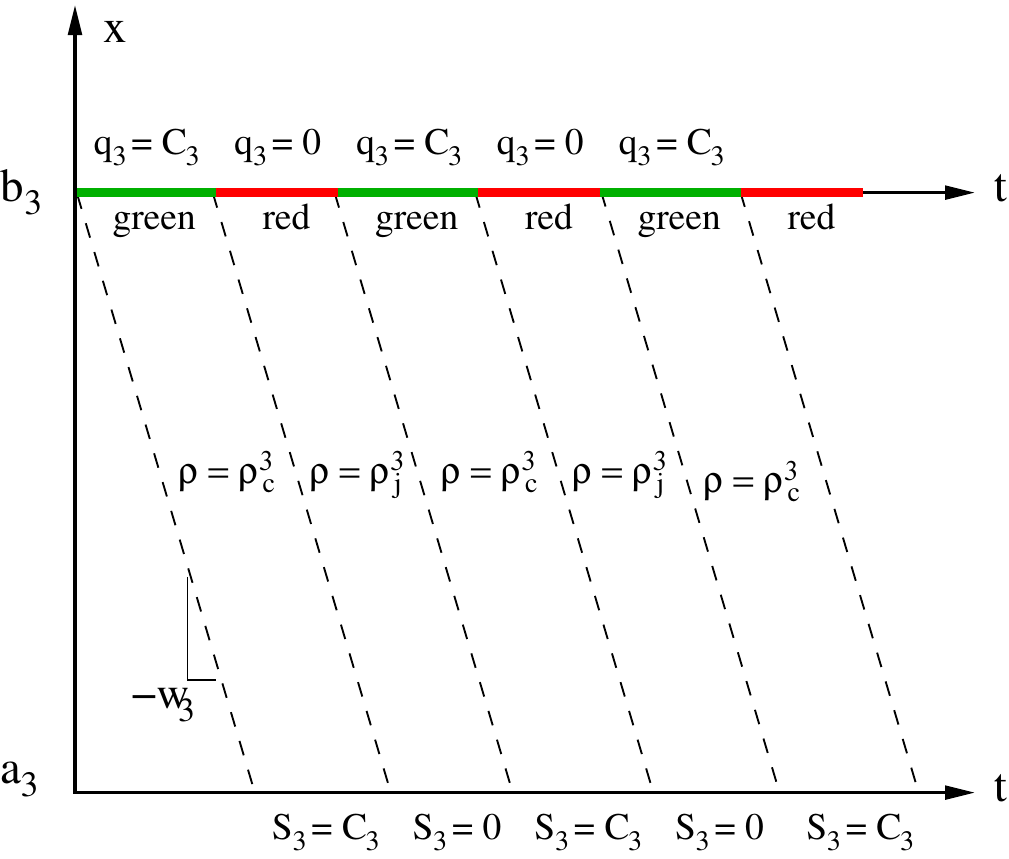}
\caption{The scenario where link $I_3$ is in the congested phase, when a triangular fundamental diagram is employed. The dashed lines represent characteristics traveling backward at speed $w_3$. $C_3$ denotes the flow capacity; $q_3$ denotes the link exit flow; $\rho_c^3$ and $\rho^3_j$ denote the critical density and the jam density respectively; $S_3$ denotes the supply at the entrance of the link.}
\label{figScenario4}
\end{figure}

The key observation is that the effective supply $\mathcal{S}_3^1(t)$ does not have bounded variation as the signal cycle of $u_3(t)$ tends to zero.  As a consequence, the convergence expressed in \eqref{weakconveqn2} no longer holds. To see this, we simply adjust $u_1$ and $u_3$ such that  $S_3(t)= C_3\cdot u_1(t)$ (in this case, we say that signal controls $u_1$ and $u_3$ are resonant) and perform the following calculation:
$$
\int_0^t\mathcal{S}_3^1(\tau)u_1(\tau)\,d\tau~=~\int_0^t \min\{C_1,\,C_3\}\cdot u_1^2(\tau)\,d\tau~=~\int_0^t \min\{C_1,\,C_3\}\cdot u_1(\tau)\,d\tau~=~\int_0^t \mathcal{S}_3^1(\tau)\,d\tau
$$
which never converges to $\displaystyle \eta_1\int_0^t \mathcal{S}_3^1(\tau)\,d\tau$ regardless of the cycle length $\Delta_A$.  Even if $u_1$ and $u_3$ are not resonant, due to the fact that $\mathcal{S}_3^1(t)$ does not have bounded variation as the cycle of $u_3(t)$ becomes smaller and smaller, the convergence  will not hold in general (in fact, $\mathcal{S}_3^1(t)$ will oscillate more and more violently as the cycle of $u_3(t)$ diminishes). We thus conclude that in the case of a triangular fundamental diagram, the proposed continuum junction model does not yield a sound approximation of networks controlled by more than one on-and-off signal lights, unless the spillback case depicted in Figure \ref{figScenario4} does not occur.

\subsubsection{The case with strictly concave fundamental diagram}

This section establishes convergence result for the on-and-off signal model with strictly concave fundamental diagram. A strictly concave fundamental diagram $f(\cdot)$ is a piecewise smooth function that satisfies, in addition to  {\bf (F)}, 
\begin{equation}\label{scdef}
f''(\rho)\leq -b \qquad \hbox{for some }~ b~>~0
\end{equation}
 for all $\rho\in[0,\,\rho_j]$ such that $f(\cdot)$ is twice differentiable at $\rho$.

Let us re-visit the scenario where link $I_3$ is dominated by the congested phase, but now assume a strictly concave fundamental diagram. We begin with the observation that in this case the characteristic field is genuinely nonlinear. As a result, any flux variation generated by signal control at the exit of the link gets instantaneously reduced when the waves propagate backwards, see \cite{Bressan 2000} for more mathematical details. Therefore, it is expected that the convergence may still hold  even in the presence of spillback. The following lemma is the key ingredient of our convergence result and its proof is quite informative.

\begin{lemma}\label{strictconcavespthm}
Given the merge junction in Figure \ref{figsimplemerge}, we focus on the link $I_3$ expressed as a spatial interval $[0,\,L]$ with a strictly concave fundamental diagram $f(\cdot)$. Assume that a signal control $u_3(t)$ with a fixed cycle-and-split is present at the exit of $I_3$ (node $B$) and that the whole link $I_3$ remains in the congested phase. Then the supply function $S_3(t)$ converges to some constant $S_3^*$ uniformly  as the cycle length of $u_3(t)$ tends to zero. 
\end{lemma}
\begin{proof}
The proof is divided into several parts.

\noindent{\bf Part 1.}  We begin by noticing that when $I_3$ is in the congested phase, in the presence of alternating phases of red and green at the downstream boundary of $I_3$, the density profile on this congested link consists of shock and rarefaction waves. As Part I of Figure \ref{figspillbackpattern} shows, during the green time, a rarefaction wave is formed which then interacts with the characteristic lines generated by the red phase that comes afterwards. As a result of such interaction, a shock is formed which propagates backward until it reaches the entrance of $I_3$ ($x=0$). It is quite obvious from Part I of Figure \ref{figspillbackpattern} that the flow variable at the entrance of $I_3$, and hence the supply $S_3(t)$, will display a repeated pattern with downward jumps, which is illustrated in Part II of Figure \ref{figspillbackpattern}. Therefore, to show the desired result it suffices to estimate the magnitude of such jumps using an Oleinik-type estimate, as follows. \\

\begin{figure}[h!]
\centering
\includegraphics[width=\textwidth]{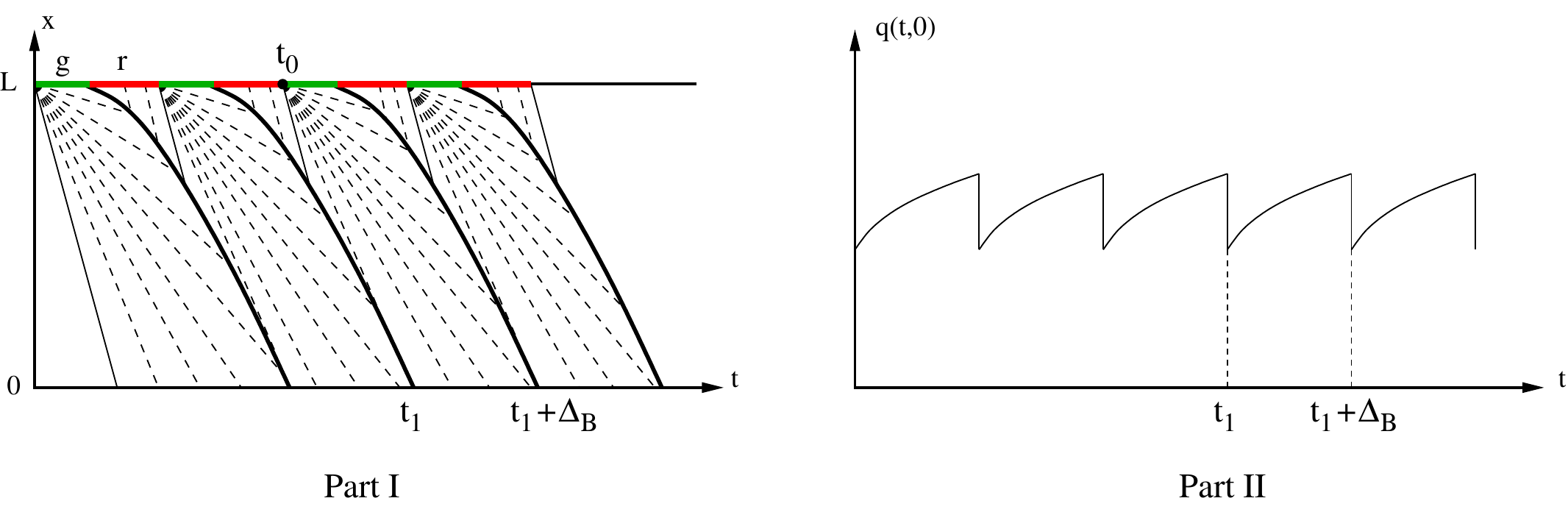}
\caption{Link $I_3$ dominated by the congested phase. Left: the repeated pattern in the density profile is caused by the on-and-off signal at the downstream boundary. Right: the time-dependent flow variable (and also the supply $S_3(t)$) observed at the upstream boundary of $I_3$.}
\label{figspillbackpattern}
\end{figure}

\noindent{\bf Part 2.} We readily notice that due to the prevailing congested phase, one can represent the dynamics with a scalar conservation law in the flow variable $q(t,\,x)$ instead of the density variable $\rho(t,\,x)$. In particular, we define 
\begin{equation}\label{gdef}
g(q)~\doteq~\max\Big\{\rho\in[0,\,\rho_j]:~f(\rho)~=~q\Big\}~=~\Big\{\rho\in[\rho_c,\,\rho_j]: ~f(\rho)~=~q  \Big\}\qquad q\in[0,\,C]
\end{equation}
which is the inverse of the fundamental diagram corresponding to the congested phase, where $\rho_c$ denotes the unique critical density and $C$ denotes the flow capacity. We introduce the conservation law with a downstream boundary condition
\begin{equation}\label{clawu}
\begin{cases}
&\partial_xq(t,\,x)+ \partial_t g\big(q(t,\,x)\big)~=~0\qquad (t,\,x)\in[0,\,T]\times[0,\,L]
\\
&q(t,\,L)~=~q_{exit}(t)  
\end{cases}
\end{equation}
where $q(t,\,x)$ denotes the flow at a point in the temporal-spatial domain, $q_{exit}(t)$ denotes the link exit flow which is determined by the signal control $u_3(t)$. The conservation law in \eqref{clawu} is equivalent to \eqref{LWRPDE} under the assumption that the link is dominated by the congested phase. A similar technique has been applied in \cite{BH}, and the reader is referred to \cite{LWRDUE} for a proof of such equivalence. We further notice that by switching the roles of $x$ and $t$, the downstream boundary condition $q_{exit}(t)$ can be viewed as a ``terminal condition" for \eqref{clawu}. Since the Oleinik estimate holds only in a time-forward fashion \citep{Bressan 2000}, we introduce the dummy variable $y=-x$ and write
\begin{equation}\label{clawu'}
\partial_yq(t,\,-y)- \partial_tg\big(q(t,\,-y)\big)~=~0\qquad (t,\,y)\in[0,\,T]\times[-L,\,0]
\end{equation}
with what is now the ``initial condition"
\begin{equation}\label{clawinicond}
q(t,\,-L)~=~q_{exit}(t)
\end{equation}
For such an initial value problem \eqref{clawu'}-\eqref{clawinicond}, the standard Oleinik estimate holds, that is, 
\begin{equation}\label{Oleinikest}
\partial_t q(t,\,-y)~\leq~{-1\over c(L+y)},\qquad\hbox{or}\qquad \partial_t q(t,\,x)~\leq~{-1\over c(L-x)}
\end{equation}
whenever $(t,\,-y)$ or $(t,\,x)$ is away from shock waves, where $c<0$ is any upper bound on $g''(\cdot)$.  \footnote{Notice that the estimate \eqref{Oleinikest} holds true for any downstream condition $q_{exit}(t)$.} The value of $c$ can be determined as follows. In view of \eqref{gdef}, we have for any $\rho\in[\rho_c,\,\rho_j]$ that
\begin{align*}
\rho~\equiv~g\big(f(\rho)\big)~\Longrightarrow~&  0~=~\big[g\big(f(\rho)\big)\big]''~=~g''\big(f(\rho)\big)\Big[f'(\rho)\Big]^2+g'\big(f(\rho)\big)f''(\rho)
\\
~\Longrightarrow~& g''\big(f(\rho)\big)~=~-{f''(\rho)\over \big[f'(\rho)\big]^3}~\leq~{b\over\big[f'(\rho_j)\big]^3}~\doteq~c
\end{align*}
where $b$ is given by \eqref{scdef}. Notice that we used the identity $g'\big(f(\rho)\big)\cdot f'(\rho)=1$ in the above deduction. Setting $x=0$, the Oleinik estimate \eqref{Oleinikest} yields the following critical result on the gradient of the flux at the entrance of $I_3$:
\begin{equation}\label{criticalresult}
\partial_t q(t,\,0)~\leq~-{\big[f'(\rho_j)\big]^3\over bL}
\end{equation}

\noindent {\bf Part 3.} We are now in a position ready to estimate the magnitude of the downward jumps depicted in Part II of Figure \ref{figspillbackpattern}. To do so, we readily notice that the duration between two consecutive jumps is comparable to a cycle length $\Delta_B$. Thus the magnitude of the jump is bounded by 
\begin{equation}\label{boundonjump}
 \Delta_B\cdot {-\big[f'(\rho_j)\big]^3\over bL}
\end{equation}
which tends to zero as the cycle length goes to zero. We thus conclude that the flow $q(t,\,0)$, and hence the supply $S_3(t)$ converges to a constant uniformly as $\Delta_B\to 0$.
\end{proof}

\begin{remark}\label{nonlinearremark}
In contrast to the triangular case, the convergence result holds for the strictly concave case even in the presence of vehicle spillback. An intuitive explanation, as we mention before, is related to the nonlinear effect caused by the strictly concave fundamental diagram. Figure \ref{figvariations1} compares the supply profiles observed at the entrance of  link $I_3$ when the whole link is in the  congested phase. As $\Delta_B\to 0$, in the triangular case the oscillation in $S_3(t)$ has the biggest amplitude and becomes more and more frequent, causing the total variation to blow up and the convergence \eqref{weakconveqn2} to fail. On the other hand, in the strictly concave case the oscillation in $S_3(t)$ is damped as it gets more and more frequent. In fact, one may easily show by \eqref{boundonjump} that the supply $S_3(t)$ has uniform bounded variation regardless of the cycle length $\Delta_B$. Thus  the convergence \eqref{weakconveqn2} continues to hold in this case.
\end{remark}

\begin{figure}[h!]
\centering
\includegraphics[width=.8\textwidth]{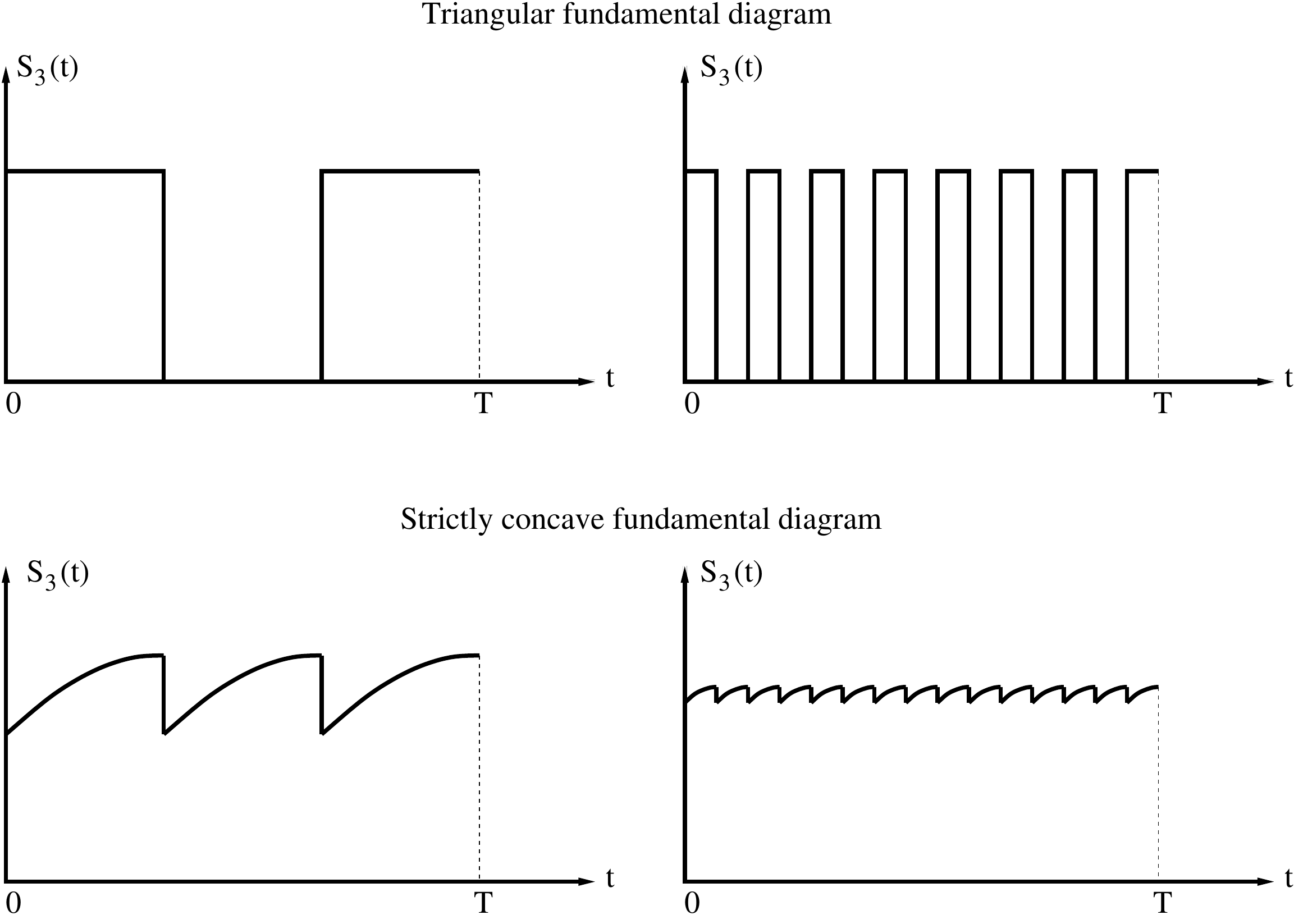}
\caption{Profiles of the supply observed at the entrance of $I_3$, when $I_3$ is  dominated by the congested phase and controlled by a signal $u_3(t)$ at the exit. First row: the triangular case; second row: the strictly concave case. First column: larger signal cycle; second column: smaller signal cycle.}
\label{figvariations1}
\end{figure}

We are now in a position ready to state and prove the convergence result for the strictly concave case. 

\begin{theorem}
Consider a  network with a fixed-cycle-and-split signal control at each node, where the flow dynamic on each link is governed by a Hamilton-Jacobi equation \eqref{HJE} with a strictly concave fundamental diagram.  We further assume that the entrance of each link may remain in the congested phase for some period of time, which spans at least several signal cycles. Then the solution of this network converges to the one corresponding to the continuum signal model, when the traffic signal cycles tend to zero.
\end{theorem}
\begin{proof}
For the junction depicted in Figure \ref{figsimplemerge}, if the entrance of $I_3$ is in the uncongested phase, the convergence is proven in Theorem \ref{thmconvns}. If $I_3$ is dominated by the congested phase, Lemma \ref{strictconcavespthm} asserts that $S_3(t)$ converges to some $S_3^*(t)$ uniformly and in the $L^2$-norm as $\Delta_B\to 0$, which implies the same convergence $\mathcal{S}_3^1(t) \to \mathcal{S}_3^{1,*}(t)$, where $\mathcal{S}_3^1(t)=\min\{C_1,\,S_3(t)\}$, $\mathcal{S}_3^{1,*}(t)=\min\{C_1,\,S_3^*(t)\}$. Thus we have
$$
\int_{0}^t \mathcal{S}_3^1(\tau)u_1(\tau)\,d\tau~~\longrightarrow~~\eta_1\int_0^t \mathcal{S}_3^{1,*}(\tau)\,d\tau,\qquad \hbox{as}\qquad \Delta_A ~\to ~0
$$
uniformly for all $t\in[0,\,T]$. We then apply a proof similar to that in Theorem \ref{thmconvns} and conclude the Moskowitz function $N^{\Delta_A}(t,\,x)$ converges to $N^0(t,\,x)$ on $I_1$  as $\Delta_A\to 0$.  

Finally, notice that the approximation error of the continuum model adds up linearly through the network, thus such convergence holds on the whole network.
\end{proof}

\subsection{Approximation errors with spillback}
This section is devoted to establishing the approximation error of the continuum signal model when the signal cycles in the on-and-off model are not infinitesimal, and when spillback occurs at the intersection. In contrast to the non-spillback case, the difference  between the two models in the presence of spillback may be larger and grow with time.

Let us recall the signalized merge junction shown in Figure \ref{figsimplemerge}, and  focus on the link $I_1$. All notations employed earlier will remain in effect in this section.

\begin{theorem}\label{estthm2}{\bf (Error estimate with spillback)}
Under the same setting and notations of Theorem \ref{estthm1}, we assume that the entrance of  link $I_3$ may remain in the congested phase for some period of time, which spans at least several signal cycles. Then if the Hamiltonian is triangular, 
\begin{equation}\label{estthm2eqntri}
\left|N^{\Delta_A}(t,\,x)-N^0(t,\,x)\right|~\leq~\eta_1(1-\eta_1)\Delta_A \min\{C_1, C_3\} +\min\{C_1,\,C_3\}\,\eta_1 t
\end{equation}
where $C_1$ and $C_3$ denotes the flow capacity of link $I_1$ and $I_3$ respectively. If the Hamiltonian is strictly concave, 
\begin{equation}\label{estthm2eqnGreen}
\left|N^{\Delta_A}(t,\,x)-N^0(t,\,x)\right|~\leq~\eta_1(1-\eta_1)\Delta_A \min\{C_1, C_3\}+\min\left\{C_1,\,f\left((f')^{-1}\left({-L\over L/w+\Delta_B}\right)\right)\right\} \eta_1 t
\end{equation}
for all $(t,\,x)\in [0,\,T]\times[a_1,\,b_1]$, where $\Delta_B$ denotes the cycle length of signal $u_3(t)$ located at the downstream boundary of $I_3$, $f(\cdot)$ denotes the strictly concave fundamental diagram of $I_3$ with $f'(\rho_j)=-w$, $\rho_j$ and $L$ denotes the jam density and  the length of $I_3$ respectively.
\end{theorem}

\begin{remark}
In case $f(\cdot)$ is not continuously differentiable at some $\rho_1$
and $f'(\rho_1-) >{-L\over L/w+\Delta_B}>f'(\rho_1+)$, then
we set 
$$
(f')^{-1}\left({-L\over L/w+\Delta_B}\right)~=~\rho_1
$$
The conclusion of Theorem \ref{estthm2} still holds.
\end{remark}

\begin{proof}

\noindent {\bf Case 1. (Triangular Hamiltonian)} When the entrance of link $I_3$ becomes congested, the scenario described in Figure \ref{figScenario4} occurs. As a result, the supply function $S_3(t)\in\{0,\,C_3\}$, which implies that $\mathcal{S}_3^1(t)\in\big\{0,\,\min\{C_1,C_3\}\big\}$.   Therefore for any $t\gg \Delta_A$,
$$
0~\leq~N_{down}^{\Delta_A}(t)~=~\int_0^t\mathcal{S}_3^1(\tau)\cdot u_1(\tau)\,d\tau~\leq~\min\{C_1, C_3\}\int_0^t u_1(\tau)\,d\tau~\approx \min\{C_1,C_3\} \eta_1  t
$$
$$
0~\leq~N_{down}^0(t)~=~\int_0^t \eta_1 \mathcal{S}_3^1(\tau)\,d\tau~\leq~\min\{C_1, C_3\}\int_0^t\eta_1 \,d\tau~=~\min\{C_1,C_3\}\eta_1t
$$
And we have that
$$
\left|N_{down}^{\Delta_A}(t) -N_{down}^0(t)\right|~\leq~\min\{C_1, C_3\}\eta_1 t
$$
Whenever the entrance of $I_3$ becomes uncongested, the additional difference between $N_{down}^{\Delta_A}(t)$ and $N_{down}^0(t)$ is always bounded by $\eta_1(1-\eta_1)\Delta_A \min\{C_1,C_3\}$ and independent of time, as we have shown in Theorem \ref{estthm1}. For any $t\in[0,\,T]$ we deduce, in the same way as in Theorem \ref{thmconvns}, that
$$
\left|N^{\Delta_A}(t,\,x)-N^{0}(t,\,x)\right|~\leq~\max_{\tau\in[0,\,t]}\left|N^{\Delta_A}_{down}(\tau)-N^0_{down}(\tau)\right|~\leq~ \big(\eta_1(1-\eta_1)\Delta_A+\eta_1t\big)\min\{C_1,C_3\}
$$

\noindent {\bf Case 2. (Strictly concave Hamiltonian)} When the entrance of  link $I_3$ is in the congested phase, our calculation in the proof of Lemma \ref{strictconcavespthm} shows that the density profile at the entrance of $I_3$ consists of shocks and rarefaction waves. In order to establish a close estimate of the magnitude of the jumps depicted in Part II of Figure \ref{figspillbackpattern}, we notice that the duration between two consecutive jumps is approximately $\Delta_B$. Define $[t_1,\,t_1+\Delta_B]$ to be the time interval between two consecutive jumps; then we have $t_1\geq t_0+L/w$ where $t_0$ denotes the time at which this rarefaction emerges. See Part I of Figure \ref{figspillbackpattern} for an illustration. Note that for every $t\in[t_0+L/w,\,t_1+\Delta_B]$ the map $t\mapsto f\left( (f')^{-1}\left({L\over t-t_0}\right)\right)$, which is given by the rarefaction wave, is a concave function in $t$. We  deduce that
\begin{align}
\label{concavity1}
\underbrace{f\left((f')^{-1}\left({-L\over t_1+\Delta_B-t_0}\right)\right)}_{X}-\underbrace{f\left((f')^{-1}\left({-L\over t_1-t_0}\right)\right)}_{Y}~\leq~&f\left((f')^{-1}\left({-L\over L/w+\Delta_B}\right)\right)-f\left((f')^{-1}\left({-L\over L/w}\right)\right)
\\
\label{concavity2}
~=~&f\left((f')^{-1}\left({-L\over L/w+\Delta_B}\right)\right)
\end{align}
Inequality \eqref{concavity1} is due to concavity and the fact that $t_1+\Delta_B-t_0\geq L/w+\Delta_B$. Notice that $X$ and $Y$ from \eqref{concavity1} are precisely the maximum and minimum of the flow variable $q(t,\,0)$ inside the time interval $[t_1,\,t_1+\Delta_B]$. Thus \eqref{concavity2} provides an upper bound on the jumps in flow and in the supply function $S_3(t)$.

Since $Y\leq S_3(t) \leq X$, we have that $\min\{C_1,\,Y\}\leq \mathcal{S}_{3}^1(t)\leq \min\{C_1,\,X\}$. Similar to {\bf Case 1}, we have the following estimates
\begin{align}
\min\big\{C_1,\,Y\big\}\, \eta_1 t~\leq~&\int_0^t\mathcal{S}_3^1(\tau)\cdot u(\tau)\,d\tau~\leq~\min\big\{C_1,\,X\big\}\, \eta_1 t
\\
\min\big\{C_1,\,Y\big\}\, \eta_1 t~\leq~&\int_0^t\eta_1\mathcal{S}_3^1(\tau)\,d\tau~\leq~\min\big\{C_1,\,X\big\}\, \eta_1 t
\end{align}
so 
\begin{align}
\left|N_{down}^{\Delta_A}(t) -N_{down}^0(t)\right|  ~\leq~& \max_{\tau\in[0,\,t]}\left|\int_0^t\mathcal{S}_3^1(\tau)\cdot u(\tau)\,d\tau - \int_0^t\eta_1\mathcal{S}_3^1(\tau)\,d\tau\right|
\nonumber\\
~\leq~& \left(\min\big\{C_1,~X\big\}- \min\big\{C_1,~Y\big\} \right)  \eta_1t ~\leq \min\big\{C_1,\,X-Y\big\}\, \eta_1 t
\nonumber \\
\label{altproof}
~\leq~&\min\left\{C_1,\,  f\left((f')^{-1}\left({-L\over L/w+\Delta_B}\right)\right) \right\}\eta_1 t
\end{align}
Again, an additional term $\eta_1(1-\eta_1)\Delta_A \min\{C_1, C_3\}$ is attached to the error when the entrance of $I_3$ is in the uncongested phase. This establishes \eqref{estthm2eqnGreen}.
\end{proof}

Notice that  in the presence of spillback, the errors associated with triangular or strictly concave fundamental diagrams both grow with time. However, the error in the strictly concave case is much smaller than the triangular case. This is quite clear from Remark \ref{nonlinearremark}  and will be numerically verified later in Section \ref{secNumerical}. Inequalities \eqref{altproofconv1}, \eqref{estthm2eqntri} and \eqref{estthm2eqnGreen} are three of the most significant expressions in this paper, as they not only provide comprehensive error estimates, but also explain the convergence/non-convergence we established earlier when the signal cycles tend to zero: that is, by setting $\Delta_A\rightarrow 0$, $\Delta_B\rightarrow 0$, the right hand side of \eqref{estthm2eqnGreen} tends to zero, while the right hand side of \eqref{estthm2eqntri} does not.

\begin{remark}
The different errors associated with the triangular (not strictly concave) fundamental diagram and the strictly concave one can be interpreted in prose as follows: the more concave a fundamental diagram is, the more nonlinear its characteristic filed is, and the more cancellation to the flow variation it causes, hence the less error we find in the approximation. Notice that when $f(\cdot)$ is twice differentiable, the error estimates can be stated in terms of the second derivative using the same Oleinik-type estimate as we did in Lemma \ref{strictconcavespthm}, that is, the magnitude of the oscillation in supply $S_3(t)$ in the presence of spillback is bounded by 
\begin{equation}\label{myquant}
\Delta_B\cdot{-\big[f'(\rho_j)\big]^3\over bL}
\end{equation}
such quantity is directly related to the approximation error. Recall that the constant $b$ is such that $f''(\rho)\leq -b,~\forall\rho$; thus it is a direct indicator of how concave the fundamental diagram is. In particular, if $f(\cdot)$ is triangular, then $b=0$ and the quantity \eqref{myquant} blows up, resulting in the maximum error possible. Such interpretation using the second derivative provides further insight into the importance of strict concavity in reducing the approximation error. 
\end{remark}

\section{Continuum approximation in the presence of  transient spillback}
\label{secTransient}
In the previous two sections, we have considered the case with and without spillbacks. In both cases, it is assumed that the uncongested phase or the congested phase persists at the entrance of link $I_3$ for a significant period of time, usually on the scale of several signal cycles. Such assumption is crucial for our analysis since the continuum signal model is one type of aggregate models which approximates the cumulative throughput of a signalized junction within at least one full signal cycle. If the spillback/non-spillback state of the system fluctuates on a much smaller time scale, say shorter than a full cycle, then the previously established convergence and error estimates may not hold. We will refer to such situation as {\it transient spillback}.

Let us illustrate the impact of transient spillback on the approximating quality of the continuum signal model using a specific example. Consider the three-incoming, one-outgoing signal junction shown in Figure \ref{figfour}. Assume vehicles flow into links $\text{a}_1$ and $\text{a}_3$ with the maximum rate (flow capacity), while link $\text{a}_2$ remains empty. Let $\text{a}_4$ be congested. In addition, assign equal signal split of 1/3 to each incoming link. We also stipulate that the sequence of links allowed to enter $\text{a}_4$ is $\text{a}_1,\,\text{a}_2,\,\text{a}_3,\,\text{a}_1,\,\text{a}_2,\,\text{a}_3\ldots $.

\begin{figure}[h!]
\centering
\includegraphics[width=.35\textwidth]{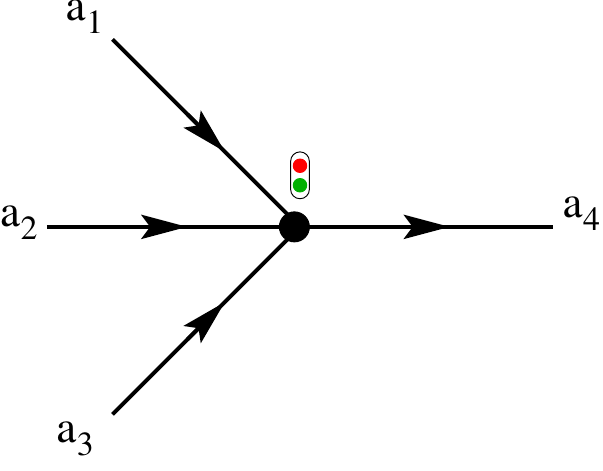}
\caption{The signalized junction with three incoming links and one outgoing link.}
\label{figfour}
\end{figure}

In an on-and-off scenario, when the light for $\text{a}_1$ is green, the supply from downstream is limited since $\text{a}_4$ is in the congested phase. On the other hand, when the light turns green for $\text{a}_3$, due to the previous green phase for the empty link $\text{a}_2$ which allows the entrance of $\text{a}_4$ to clear up a little bit, the supply from $\text{a}_4$ will be maximum and equal to its flow capacity. When vehicles from $\text{a}_3$ fill up this empty space on link $\text{a}_4$, vehicles from $\text{a}_1$ are once again faced with a limited downstream supply. Therefore, it is reasonable to expect the throughputs of links $\text{a}_1$ and $\text{a}_3$ are quite different, despite the fact that they are assigned equal signal split. Such an asymmetric situation will persist even if signal cycles tend to zero.

A numerical simulation is conducted to confirm this observation. Figure \ref{figasymmetric} shows the throughputs (cumulative exiting vehicle curves) of links $\text{a}_1$ and $\text{a}_3$, when then same signal split of 1/3 is assigned to each approach. In these figures, the bifurcation point of the cumulative curves indicates the first time spillback occurs. We clearly observe that convergence of the on-and-off model to the continuum model does not hold, no matter how small the signal cycle is. Figure \ref{figasysupply} shows the supply profiles on link $\text{a}_4$, where we observe the predicted transient spillbacks (where the supply is lower). Such transient spillback resonates with the signal phases, causing links $\text{a}_1$ and $\text{a}_3$ to face completely different downstream supplies when their respective lights are green.

\begin{figure}[h!]
\centering
\includegraphics[width=\textwidth]{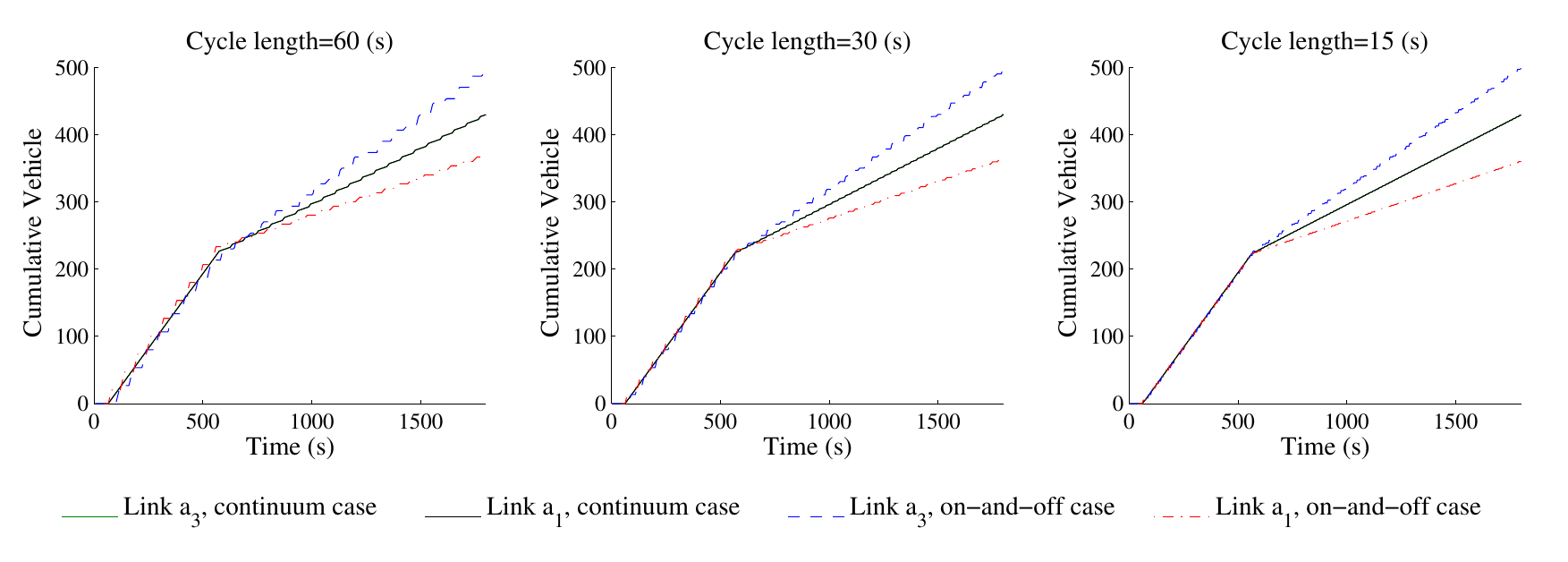}
\caption{The cumulative exiting vehicle counts on links $\text{a}_1$ and $\text{a}_3$, with the continuum model and the on-and-off model.}
\label{figasymmetric}
\end{figure}

\begin{figure}[h!]
\centering
\includegraphics[width=\textwidth]{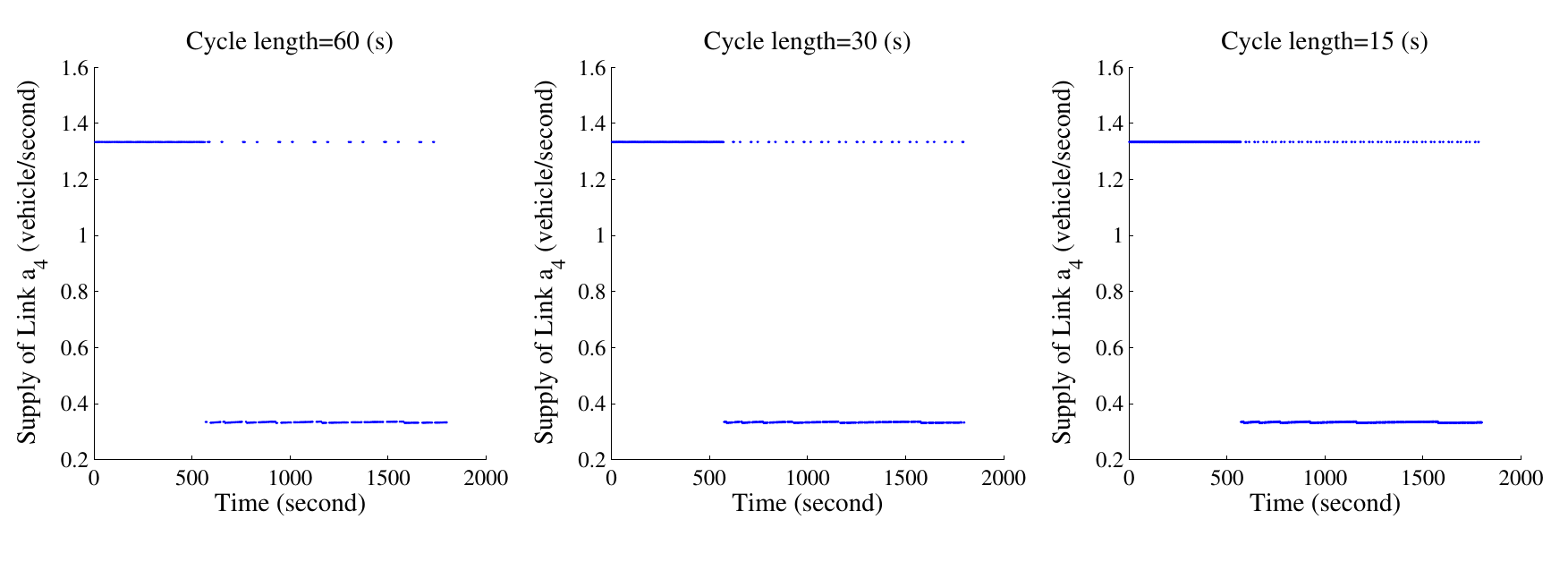}
\caption{The supply functions of $\text{a}_4$. The lower value ($\approx$ 0.33) indicates the  presence of spillback; the higher value ($=$ 4/3), which is equal to the flow capacity, indicates the absence of spillback. The fast switching between these two values implies the presence of transient spillback.}
\label{figasysupply}
\end{figure}

The computation presented above employs the LWR model with a Greenshields fundamental diagram \citep{Greenshields}. However, it is not difficult to conclude that such non-convergence will hold for any type of fundamental diagram. Such example reveals a technical difficulty arising from the theoretical investigation of the continuum approximation, although one may argue that modeling these phenomena exactly loses importance in large scale applications. Therefore, this should not completely diminish the value of the continuum signal model in the venue of engineering applications.

\section{Numerical Study}\label{secNumerical}
The goal of this section is to numerically verify the convergence results and error analysis established in the previous sections. Let us again focus on the merge node depicted in Figure \ref{figsimplemerge}, with signal controls at the exits of $I_1,\,I_2$ and $I_3$. Two types of fundamental diagrams are considered in this numerical study: the triangular fundamental diagram \eqref{triangularfd} and the Greenshields fundamental diagram \citep{Greenshields}
\begin{equation}\label{greenshieldsdef}
f(\rho)~=~\rho v_0\left(1-{\rho\over\rho_j}\right)
\end{equation}
where $v_0$ denotes the free-flow speed, $\rho_j$ denotes the jam density. Link parameters related to these two fundamental diagrams are given in Table \ref{tabparameter}. For simplicity, all  three links are assumed to have the same parameters.

\begin{table}[h!]
\centering
\begin{tabular}{|c|c|c|c|c|}\hline
                & Free-flow speed    & Jam density   & Critical density   & Flow capacity
\\                
                & (meter/second)    &  (vehicle/meter)   & (vehicle/meter)& (vehicle/second)
\\                
\hline                 
 Triangular fd  & 40/3  & 0.4  & 0.1  & 4/3
 \\
 Greenshields fd  &  40/3  & 0.4  & 0.2  & 4/3  \\
\hline

\end{tabular}
\caption{Link parameters}
\label{tabparameter}
\end{table}

\subsection{Without spillback}
Assume that the entrance of $I_3$ remains in the uncongested phase so that spillback does not occur. In addition, we let the signal control $u_1(t)$ for $I_1$ satisfy: $\Delta_A=60$ seconds,  $\eta_1=0.5$. Thus the theoretical error bound given by Theorem \ref{estthm1} is $\eta_1(1-\eta_1)\Delta_A C_3=20$ (vehicles). The Moskowitz functions for $I_1$ are shown  in Figure \ref{fignospillback6}, where both the on-and-off signal model and the continuum signal model are employed. It is clearly observed that the continuum signalized junction model yields very good approximation to the one with the on-and-off signal controls, for both triangular and Greenshields fundamental diagrams. In particular,  the absolute differences in the Moskowitz functions are uniformly bounded by 20 (vehicles) and are independent of time, which coincides with the theoretical result established in Theorem \ref{estthm1}. We are also assured that such errors are almost unobservable using the normal scales of the Moskowitz functions.

\begin{figure}[h!]
\centering
\includegraphics[width=\textwidth]{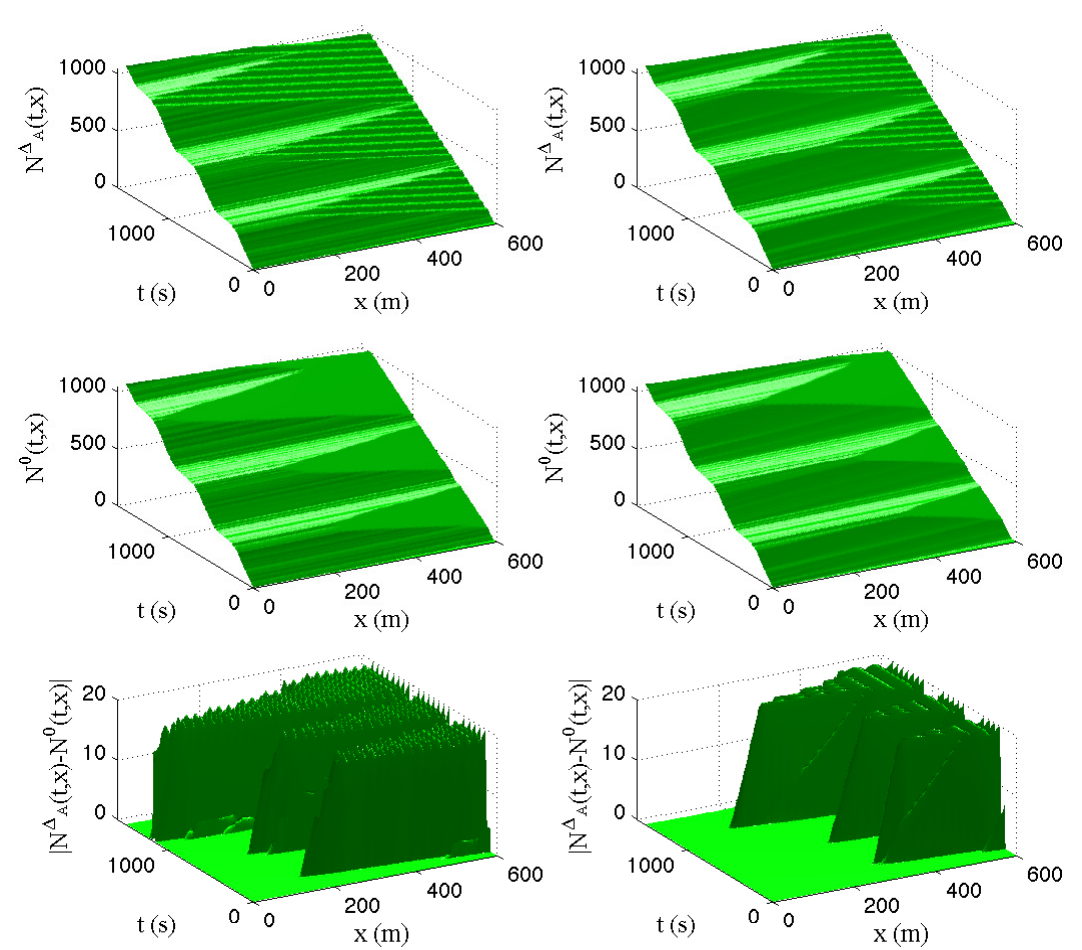}
\caption{Moskowitz functions on link $I_1$ with the triangular fundamental diagram (left column) and the Greenshields fundamental diagram (right column). The first row corresponds to the on-and-off signal model; the second row corresponds to the continuum signal model; the third row presents the absolute differences in the Moskowitz functions from the first two rows.}
\label{fignospillback6}
\end{figure}

\subsection{With spillback}

\subsubsection{Triangular fundamental diagram}
Assume that  link $I_3$ is dominated by the congested phase, then the supply function $S_3(t)$ is illustrated in Figure \ref{figScenario4}. Let us now examine the difference $|N_{down}^{\Delta_A}(t)-N_{down}^0(t)|$ for link $I_1$. First notice that it is entirely possible that whenever the control $u_1(t)=1$, the supply $S_3(t)=0$. In this case we have $N_{down}^{\Delta_A}(t)\equiv 0$ since no car can go through, while $N_{down}^0(t)$ is proportional to the integral of $S_3(t)$. Thus huge error is expected in this case. This  is illustrated in Figure \ref{figspillbacktri} with two different values of $\eta_1$. We see from these figures that when spillback occurs, the errors grow with time roughly linearly and are within the theoretical bounds provided by \eqref{estthm2eqntri}. We also observe from the upper right picture that the established error bounds are tight, as they can be approached in some actual cases.

\begin{figure}[h!]
\centering
\includegraphics[width=\textwidth]{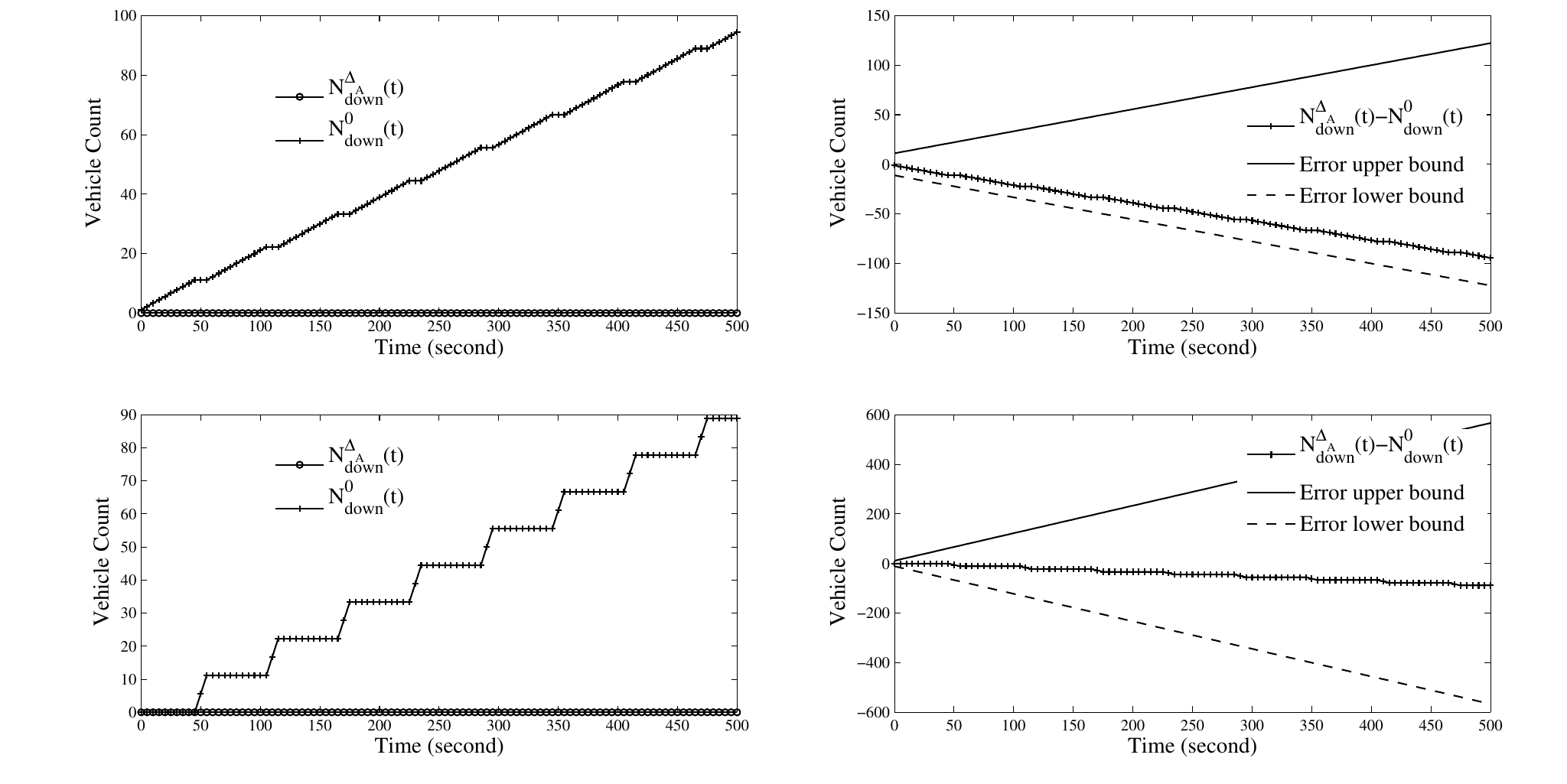}
\caption{Comparison of $N_{down}^{\Delta_A}(t)$ and $N_{down}^0(t)$ under a specific scenario where $u_1(t)=1 \Longrightarrow S_3(t)=0$. In this case $N_{down}^{\Delta_A}(t)\equiv 0$. First row: $\eta_1=1/6$. Second row: $\eta_1=5/6$.}
\label{figspillbacktri}
\end{figure}

\subsubsection{Smooth and strictly concave fundamental diagram}
Let us turn to the case with a strictly concave fundamental diagram, for instance, the Greenshields fundamental diagram. The damping effect on the flow variations caused by the strictly concave Greenshields fundamental diagram is demonstrated in Figure \ref{figspillbackTV1}, where the supply at the entrance of $I_3$ is plotted with different choices of parameters. Notice that when the fundamental diagram $f(\cdot)$ takes the explicit form of \eqref{greenshieldsdef}, the upper bound  on the jumps given by \eqref{concavity2} is expressed more accurately as follows \citep{HPGFY}:
\begin{equation}\label{GSjumpbound}
{\rho_j v_0^2\Delta_B\over 4}{2L+v_0\Delta_B\over (L+v_0\Delta_B)^2}
\end{equation}

It is clearly observed from Figure \ref{figspillbackTV1} that while the exit flow of $I_3$ can have a big variation due to the presence of the signal $u_3(t)$, the supply function on the other end of the link has a reduced variation,   in particular when the cycle $\Delta_B$ decreases and when the length $L$  increases. This is consistent with \eqref{GSjumpbound}.  To further verify the bounds, Table \ref{tabjumpbd} below compares the computed jumps in supply with the theoretical bound \eqref{GSjumpbound}. From this table we see that the approximation error conveyed by inequality \eqref{estthm2eqnGreen} (a) is valid and correct; and (b) provides a tight bound of the error.

\begin{figure}[h!]
\centering
\includegraphics[width=\textwidth]{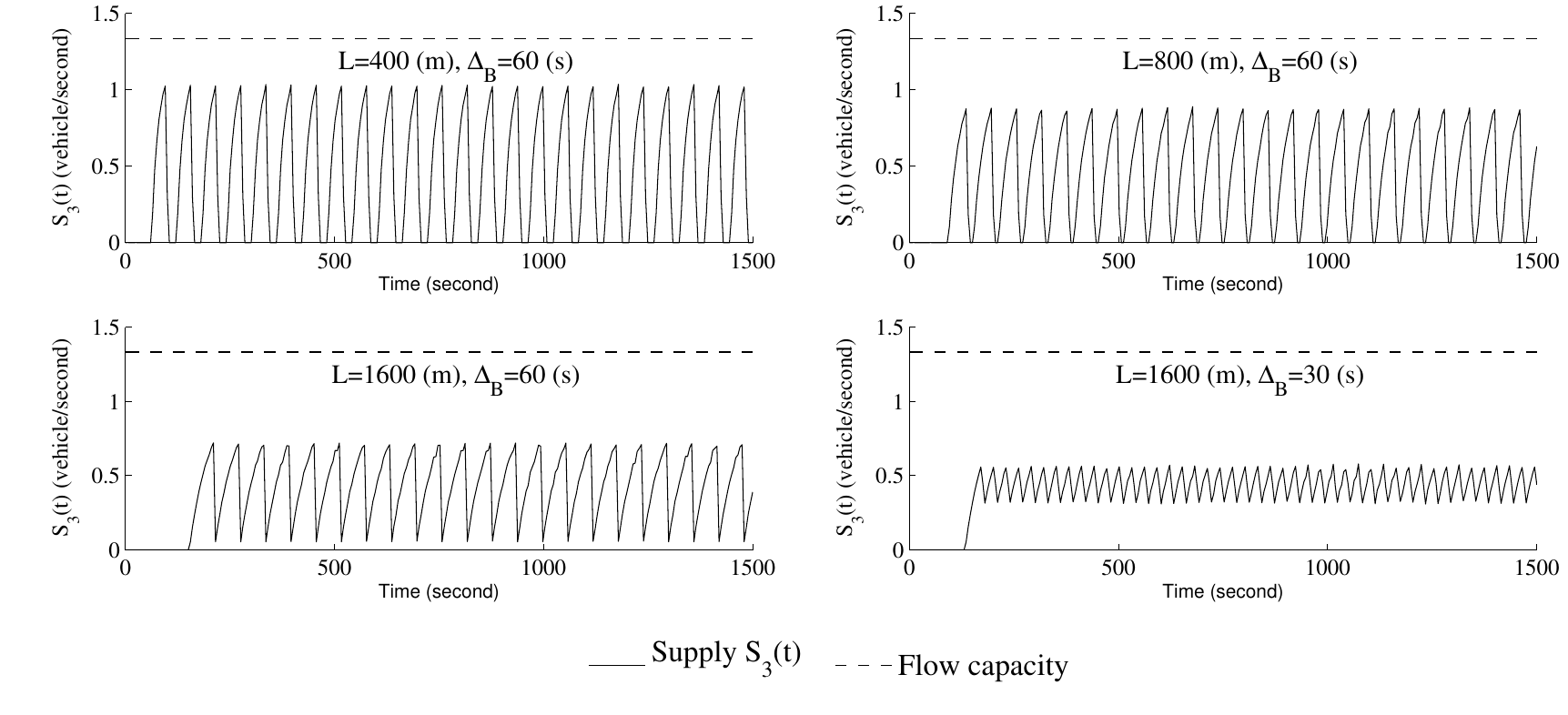}
\caption{The supply function $S_3(t)$ at the entrance of link $I_3$, when $I_3$ is dominated by the congested phase. The signal split for $I_3$ is fixed to be 1/3. The broken lines indicate the link flow capacity. Several values of link length $L$ and signal cycle $\Delta_B$ are considered, and the jumps in the supply are consistent with the theoretical bound \eqref{concavity2}.}
\label{figspillbackTV1}
\end{figure}
\begin{table}[h!]
\centering
\begin{tabular}{|c|c|c|c|c|}\hline
                & $L~=~400$ (m)   & $L~=~800$ (m)   & $L~=~1600$ (m)  & $L~=~1600$ (m)
\\                
                & $\Delta_B~=~60$ (s)    &  $\Delta_B~=~60$ (s)   & $\Delta_B~=~60$ (s)  & $\Delta_B~=~30$ (s)
\\                
\hline                 
Computed jump  & 1.04  & 0.89  & 0.67  & 0.43
 \\
 Theoretical bound  &  1.19  & 1.0  & 0.74  & 0.48  \\
\hline

\end{tabular}
\caption{Comparison of the actual jumps in the supply profile $S_3(t)$ with the theoretical bound on the jumps, when link $I_3$ is completely in the congested phase.}
\label{tabjumpbd}
\end{table}

In Theorem \ref{estthm2} we showed that the continuum approximation error grows with time when spillback occurs, and that the error increases roughly linearly for both the triangular fundamental diagram and the strictly concave fundamental diagram, see \eqref{estthm2eqntri} and \eqref{estthm2eqnGreen}. In order to numerical verify such results, we consider the merge junction from Figure \ref{figsimplemerge} with link $I_3$ in the congested phase so that spillback occurs at intersection $A$. The traffic signal at intersection $B$ has a cycle of $60$ seconds and a split ration of $1/2$ for $I_3$. The length of link $I_3$ is $400$ meters. We set the signal cycle at $A$ to be $60$ seconds, and experiment with two different split values for $I_1$: $\eta_1=1/3$ and $\eta_1=2/3$. Two cases with respectively the triangular fundamental diagram and with the Greenshields fundamental diagram, whose numerical specifications are provided in Table \ref{tabparameter}, are computed and the results are shown in Figure \ref{figlineargrowth}. The results indeed confirms that when spillback happens, both cases have errors that grow linearly with time. Moreover, the Greenshields case yields a smaller error, which coincides with out theoretical findings.

\begin{figure}[h!]
\centering
\includegraphics[width=\textwidth]{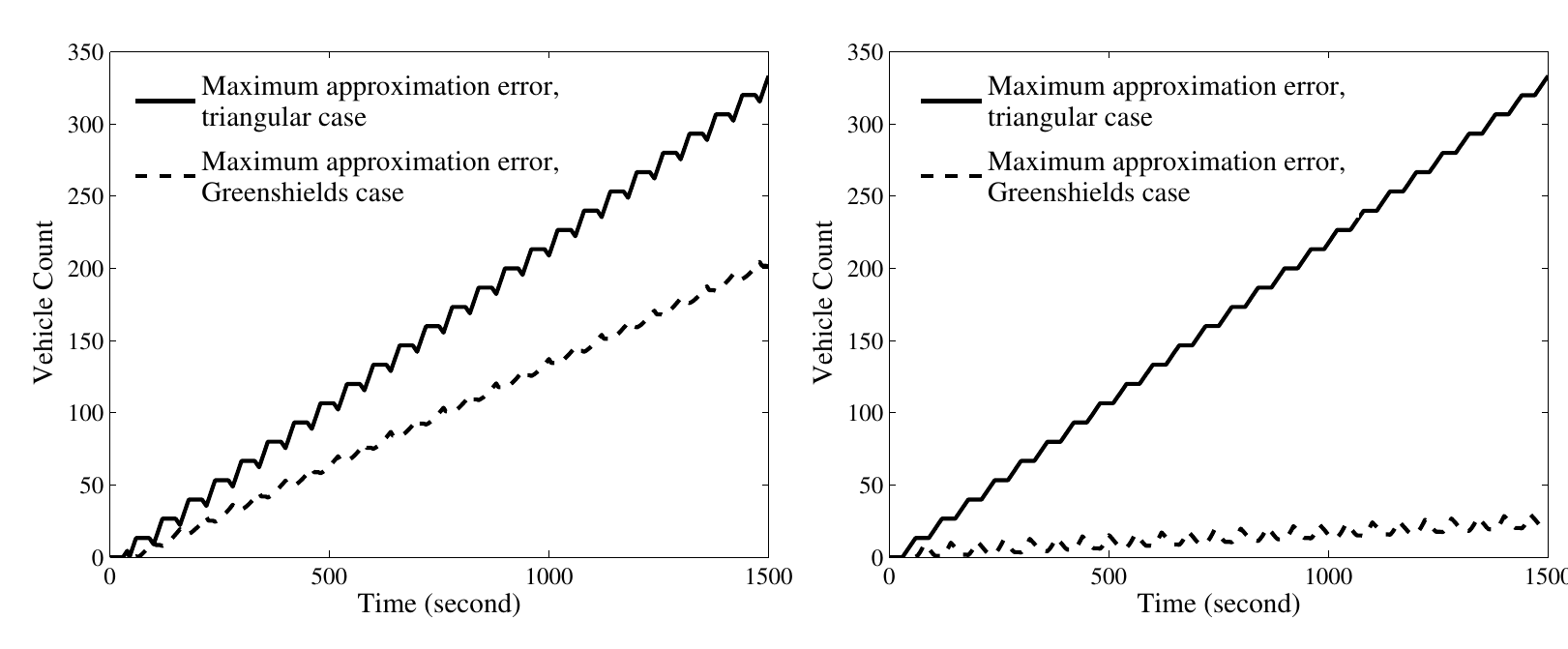}
\caption{Linear growth of the maximum approximation error (MAE): $\max_{t, x}\left|N^{\Delta_A}(t,\,x)-N^0(t,\,x)\right|$, when spillback occurs at the intersection. The triangular case and the Greenshields case are compared. Left: $\eta_1=1/3$. Right: $\eta_1=2/3$. The case with the Greenshields fundamental diagram has a smaller error.}
\label{figlineargrowth}
\end{figure}

Finally, in Figure \ref{figspillbackMaxerror} we show the maximum absolute difference between $N^{\Delta}(t,\,x)$ and $N^0(t,\,x)$ of link $I_1$ for all $(t,\,x)\in[0,\,T]\times[a_1,\,b_1]$, when the downstream link $I_3$ is entirely congested. We show such error with both triangular and Greenshields fundamental diagrams and for different values of signal cycle $\Delta_A$ and split $\eta_1$. Compared to the triangular fundamental diagram, the strictly concave (Greenshields) one yields a much lower error since the supply function $S_3(t)$ has a smaller variation due to the nonlinear effect discussed in Remark \ref{nonlinearremark} and verified in Figure \ref{figspillbackTV1}.

\begin{figure}[h!]
\centering
\includegraphics[width=\textwidth]{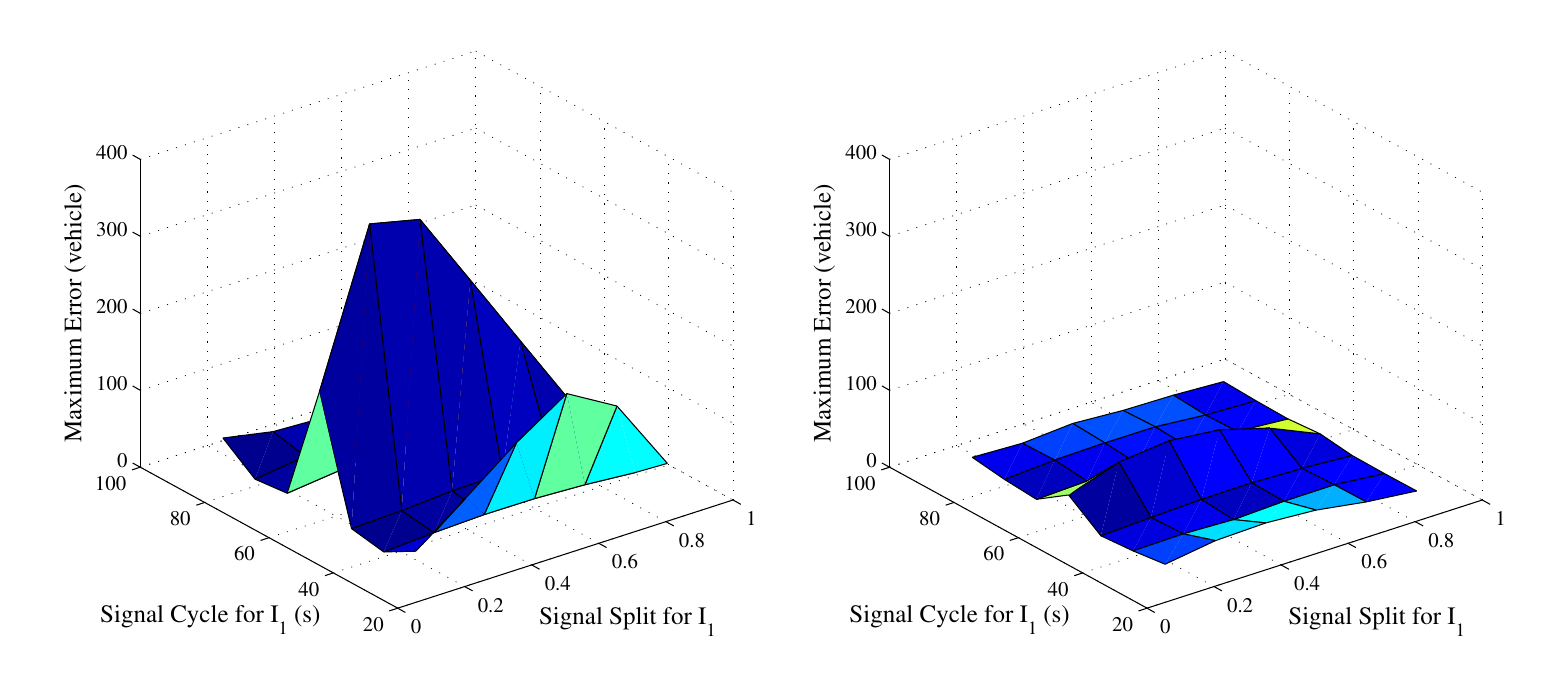}
\caption{Maximum difference  $|N^{\Delta_A}(t,\,x)-N^0(t,\,x)|,\, t\in[0,\,1500]$ (second) for link $I_1$, when the signal control for $I_3$ satisfies $\Delta_B=60$ (second), $\eta_3=0.67$ and when $I_3$ is entirely in the congested phase (spillback occurs). Left: the triangular fundamental diagram. Right: the Greenshields fundamental diagram.}
\label{figspillbackMaxerror}
\end{figure}

\section{Application}\label{secApp}

\subsection{Mixed integer linear programming approach for optimal signal control}\label{secMILP}
In a signal optimization process usually realized by mixed integer mathematical programs, usage of the  continuum signal model has several distinct advantages over the on-and-off one, such as those mentioned at the introductory part of this paper. This section presents a concrete example that demonstrates such advantages. We will provide two mixed integer linear programing (MILP) formulations using the continuum and the on-and-off signal models respectively, that aim at optimizing the dynamic network profile with proper constraints. Unlike many existing approaches that employ a cell-based dynamic, we consider a link-based kinematic wave model \citep{LKWM},  also known as the {\it link transmission model} \citep{LTM}, in order to reduce the number of (integer) variables involved in the program. These MILP formulations will not be elaborated here but are instead moved to the Appendix. A somewhat more comprehensive discussion of the MILP formulation is available in \cite{MIPsignal}.

\subsection{Numerical experiment}\label{secexperiment}

In this section, we will solve the two mixed integer linear programs (MILP) using the on-and-off signal model and the continuum signal model respectively on the same traffic network. Performances of these two MILPs and their outcomes will be compared, which illustrates the advantages of using the continuum signal model over the on-and-off one.

We consider the network depicted in Figure \ref{figComparisonntw} with three signalized intersections $A$, $B$ and $D$. Note that node $C$ is a diverge junction with no conflict of flows, therefore a signal control is not present. Traffic dynamics on each link are governed by the LWR model with a triangular fundamental diagram \footnote{The demonstrated disadvantage of using the triangular fundamental diagram in the continuum signal model is circumvented by explicitly imposing in the programs that spillback does not occur at any junction. The reason is that: 1) a non-spillback situation is reasonable  to maintain in a signal optimization process; and 2) the continuum model yields a good approximation of the on-and-off model in the absence of spillback.}. All the links in the network are assumed to have the same attributes as given in Table \ref{tabparameter}. In addition, the length of each link is set to be 400 meters.

The signal cycle length and the time step in the on-and-off model is fixed to be $60$ seconds and $10$ seconds, respectively. In other words, signal control within each cycle is determined by six binary variables. For practical reasons, we stipulate that the green time and the red time must be no less than 20 seconds, so that the signal split variable can take on only three values: $1/3$, $1/2$, and $2/3$. Moreover, in order to adapt the signal controls to a dynamic decision environment, we allow the signal splits in both the on-and-off case and the continuum case to change every 5 minutes.

\begin{figure}[h!]
\centering
\includegraphics[width=.5\textwidth]{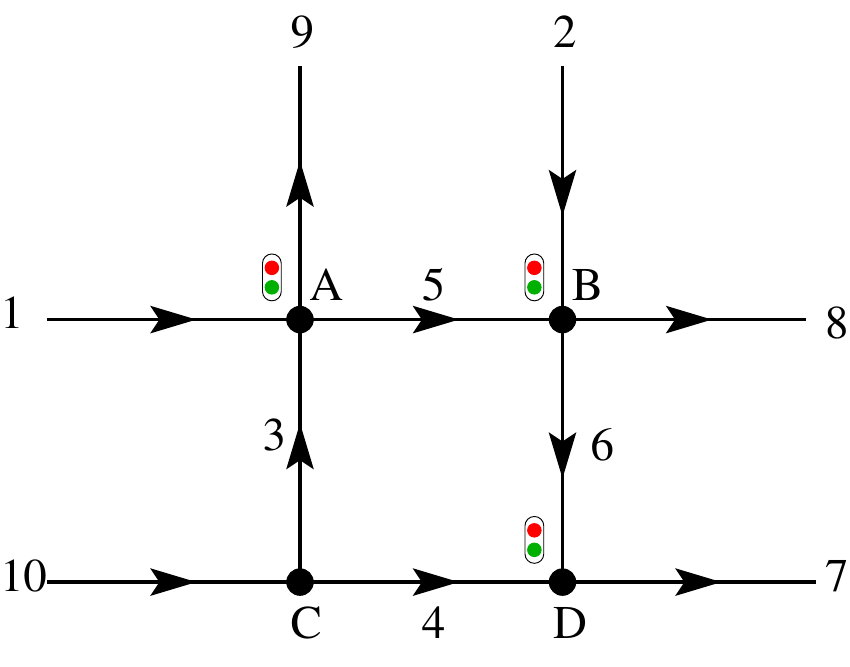}
\caption{The test network where two MILPs are performed.}
\label{figComparisonntw}
\end{figure}

\subsubsection{Performances of the two MILPs}

The inflows into the test network are randomly generated and remain the same for all the computation scenarios mentioned below.  Given that decisions on signal splits are made for every 5 minutes, we solve the MILPs on the network for time periods of 5,\,10,\,15,\, and 20 minutes. The computational times are presented in the first and second rows of Table \ref{tabtimecompare} \footnote{All the MILPs were solved with CPLEX on the Penn State High Performance Computing Systems; see http://rcc.its.psu.edu/resources/hpc/ for more details.}. Furthermore, we increase the time step size in the continuum case from 10 seconds to 30 seconds; this can be done since the time step in the continuum model is not constrained by the signal cycle or the split, which is in contrast to the on-and-off case. Results in such scenario is presented in the third row of Table \ref{tabtimecompare}. Notice that with a larger time step $h=30$ seconds,  the continuum-based MILP is capable of solving for a much larger time period, namely, one hour. This is shown in the last column of Table \ref{tabtimecompare}. In Table \ref{tabvariablecompare}, we summarize some basic information of the MILPs such as the number of continuous or binary variables involved. The results again demonstrate the computational efficiency obtained by considering the continuum signal control.

\begin{table}[h!]
\centering
\begin{tabular}{|c|c|c|c|c|c|}\hline
    Time span        &  5 min    &  10 min & 15 min  & 20  min & 60 min
\\                
\hline                 
 On and off ($h=10$ s)  & 0.44 s & 57.23  s & 249.14 s &  -  & -
 \\
 Continuum ($h=10$ s) &  0.25 s  & 9.50 s &  49.42 s   &  241.32 s  & -
 \\
 Continuum ($h=30$ s)  &  0.75 s  &  1.91 s & 3.17 s   &  9.66  s  & 130.93 s \\
\hline
\end{tabular}
\caption{Comparison of computational times of the MILPs. $h$ denotes the time step size. ``-" means that the MILP was not solved within the prescribed limit on computational time or memory usage.}
\label{tabtimecompare}
\end{table}

\begin{table}[h!]
\centering
\begin{tabular}{|c|c|c|c|}\hline
  \#  of CVs (\# of BVs)         &  5 min     & 15 min   & 60 min
\\                
\hline                 
 On and off ($h=10$ s)  & 900 (600)  &  2700 (1800)  & 10800 (7200)
 \\
 Continuum ($h=30$ s)  &  400 (200)    &  1200 (600)    & 4800 (2400) \\
\hline
\end{tabular}
\caption{Comparison of the number of continuous variables (CV) and the number of binary variables (BV) in the MILPs. $h$ denotes the time step size.}
\label{tabvariablecompare}
\end{table}

\subsubsection{Solution quality}
In Table \ref{tab90slns} we show, for a time period of 15 minutes, the optimal signal splits in the on-and-off case and in the continuum case provided by the two MILPs. We observe not only very different signal strategies in both cases, but also splits in the continuum case that are nontrivial and difficult to accommodate by the on-and-off signal model.

\begin{table}[h!]
\begin{center}
\begin{tabular}{c|c|c|c|c|c|c|}\cline{2-7}

& \multicolumn{2}{|c|}{0 - 5 min} & \multicolumn{2}{|c|}{5 - 10 min}  & \multicolumn{2}{|c|}{10 - 15 min}  
\\
\cline{2-7}
&  OAO  & Cont  & OAO  & Cont  & OAO  & Cont 
\\
\hline
\multicolumn{1}{|c|}{\multirow{1}{*}{link 1}} &  2/3  &  0.5  &  2/3  &  0.5  &  1/2  &  0.6957  
\\
\hline
\multicolumn{1}{|c|}{\multirow{1}{*}{link 2}} &  2/3  &  0.6667  &  2/3  &  0.6667  &  2/3  &  0.7
\\
\hline
\multicolumn{1}{|c|}{\multirow{1}{*}{link 3}} &  1/3  &  0.5  &  1/3  &  0.5  &  1/2  &  0.3043  
\\
\hline
\multicolumn{1}{|c|}{\multirow{1}{*}{link 4}} &  1/2  &  0.6333  &  1/3  &  0.6333  & 1/3   &  0.3  
\\
\hline
\multicolumn{1}{|c|}{\multirow{1}{*}{link 5}} &  1/3  &  0.3333  &  1/3  &  0.3333  &  1/3  &  0.3  
\\
\hline
\multicolumn{1}{|c|}{\multirow{1}{*}{link 6}} &  1/2  &  0.3667  &  2/3  &  0.3667  &  2/3  &  0.7 
\\
\hline

\end{tabular}
\end{center}
\caption{Signal split for each link as a result of the MILPs with the on-and-off (OAO) signal control and the continuum (Cont) signal control.}
\label{tab90slns}
\end{table}

In order to verify the approximation accuracy of the continuum signal model, we conduct the following calculation. The signal splits shown in Table \ref{tab90slns} corresponding to the on-and-off case are taken as given parameters to simulate the dynamic signalized network using both on-and-off and continuum models. The respective network throughputs, expressed by the cumulative exiting vehicle counts on links 7, 8, and 9, are compared in Figure \ref{figThroughputs}. Notice that for the continuum model, we employ both a smaller time step (10 seconds) and a larger time step (30 seconds) for comparison with the on-and-off case. The differences between these network throughputs are consistent with the established theoretical bounds, indicating the effectiveness of the continuum model in approximating the on-and-off model in the absence of vehicle spillback. In particular, we see that the continuum model yields a good approximation of the on-and-off model even when the time step increases significantly (30 seconds). This is because the error estimates established in Theorem \ref{estthm1} is in continuous time and independent of the time step selected for discrete-time computations. Such fact further illustrates the robustness of the continuum signal model.

\begin{figure}[h!]
\centering
\begin{minipage}{\textwidth}
\centering
\includegraphics[width=\textwidth]{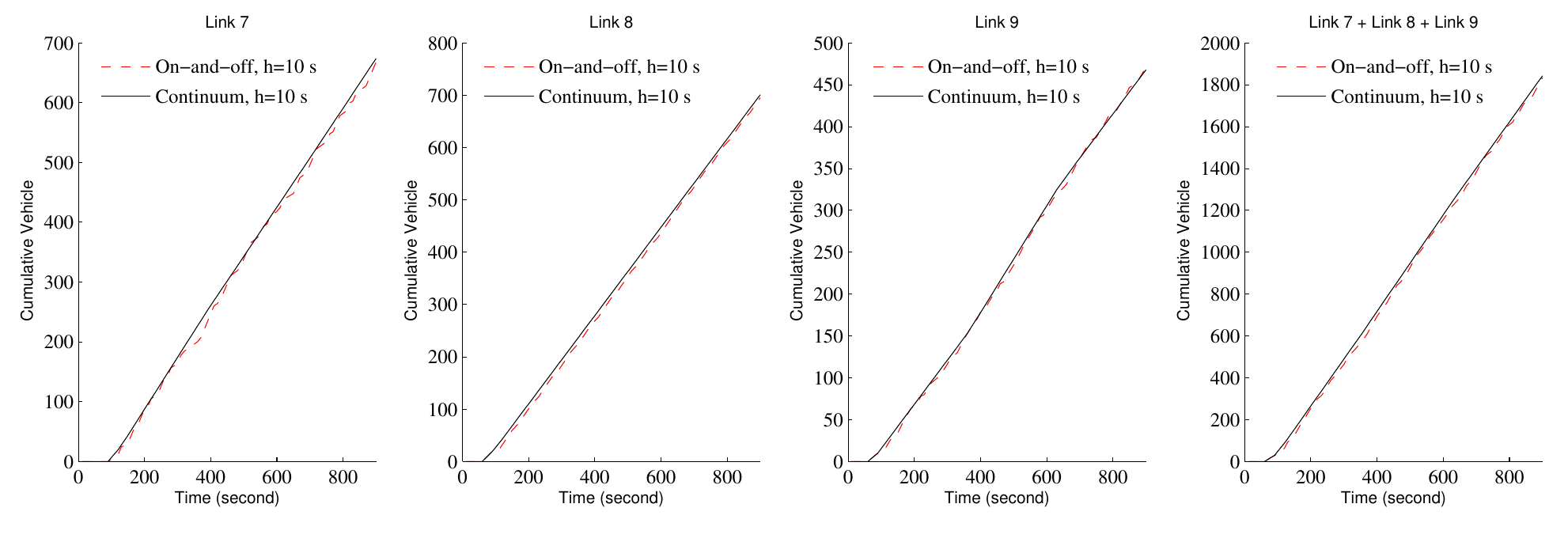}
\end{minipage}
\begin{minipage}{\textwidth}
\centering
\includegraphics[width=\textwidth]{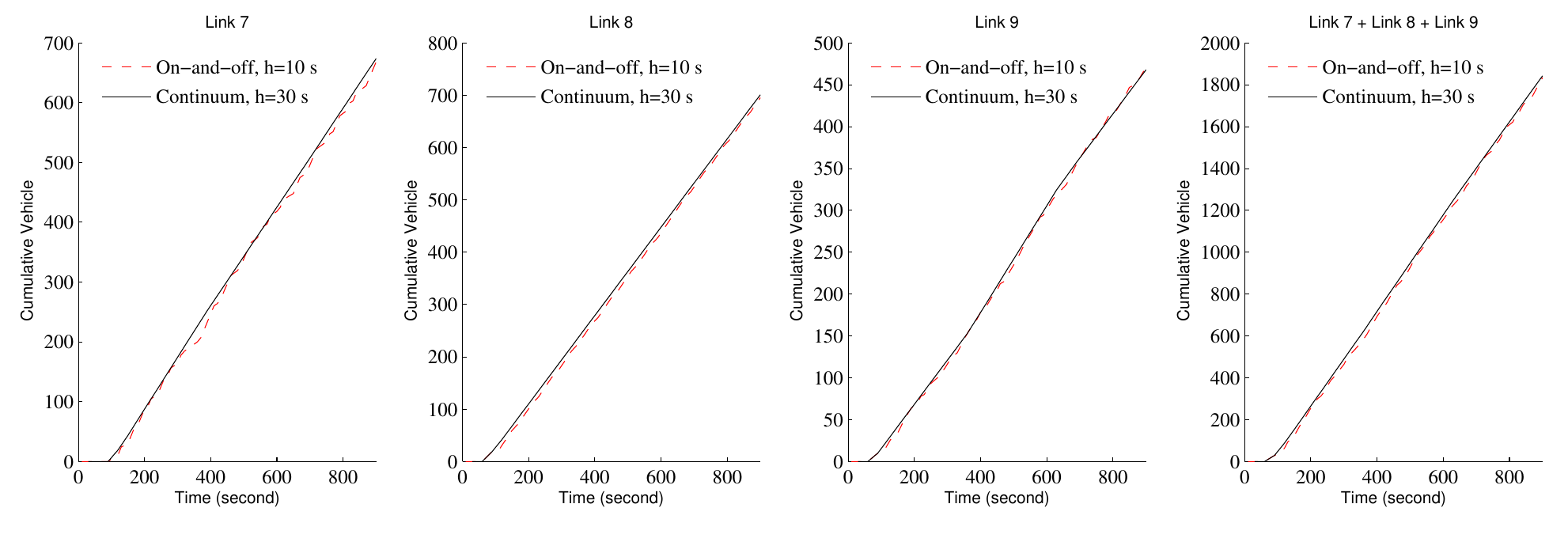}
\end{minipage}
\caption{Comparisons of the cumulative exiting vehicle counts, which are obtained by using the on-and-off signal model and the continuum signal model with the same split parameters. First row: the on-and-off case vs. the continuum case with a time step of 10 seconds. Second row: the on-and-off case vs. the continuum case with a time step of 30 seconds.}\label{figThroughputs}
\end{figure}

\section{Concluding remarks}
\label{secConclusions}
This paper is concerned with a continuum signalized junction model as an approximation of the on-and-off signal model. We provide comprehensive theoretical and numerical results on the asymptotic behavior of the on-and-off signal model and its convergence to the corresponding continuum counterpart as the signal cycle tends to zero. We also provide estimations of the difference between the two types of signal models with non-infinitesimal cycles under various scenarios. The main findings and their implications can be summarized as follows. 1) The continuum signal model with any type of fundamental diagram yields a good approximation of the on-and-off model in the absence of spillback on the signalized network. 2) When spillback occurs somewhere in the network, the continuum approximation may or may not be accurate, depending on the fundamental diagram: if a triangular fundamental diagram is used, then the continuum model does not yield a good approximation; if the fundamental diagram is strictly concave, the continuum model approximates the on-and-off model relatively well and induces much smaller error. Note that even small errors can still be significant in some cases; e.g., when blocks are very short, a maximum error of 20 vehicles could be enough to fill an entire block. Thus, it is up to the analyst to determine if the continuum approximation is appropriate for a particular scenario. The error bounds provided here can be used to make this determination. As shown, these bounds depend on the cycle length and the capacity of the links impacted. In general, intersections with lower capacity links or smaller cycle lengths would be more appropriate for the continuum approximation model.

When spillback happens, the approximation error grows with time no matter what type of fundamental diagram is employed, and that error is closely related to `how concave' the fundamental diagram is. Our technical interpretation of concavity in the fundamental diagram can be streamlined as follows in prose: the more concave a fundamental diagram is, the more nonlinear its characteristic filed is, and the more cancellation to the flow variation it causes, hence the less error we find in the approximation.

It should be noted that our results regarding the approximation efficacy of the continuum model in the presence of spillback are given in a quantitative way; that being said, we make no direct implication of which fundamental diagram is `good' or `bad' in the implementation of the continuum signal model. Rather,  a specific fundamental diagram should be evaluated with, in addition to the error bounds provided in Section \ref{secwithoutspillback} and \ref{secwithspillback},  other modeling or computational considerations and specific application scenarios.

Based on our technical results, we make the following inferences without formal proofs.

\begin{itemize}
\item In deriving the convergence results and the error estimates in the presence of spillback, the assumption of strict concavity needs only apply to the congested branch of the fundamental diagram. In other words, one can choose the uncongested branch arbitrarily as long as the minimum requirements {\bf (F)} are satisfied, and the established results still hold. For example, one may consider a piecewise-defined fundamental diagram with a linear uncongested branch and a strictly concave congested branch. 

\item If a fundamental diagram has a piecewise linear congested branch, then the more the linear pieces, the less error the continuum approximation induces when spillback occurs. This can be explained rather intuitively by the wave-front tracking algorithm \citep{Dafermos, GP}. Nevertheless, the convergence result may not hold for the piecewise linear fundamental diagram in the presence of spillback.

\item The continuum signal model and our methodological framework are easily generalizable to a signal junction with $m$ incoming links and $n$ outgoing links, where $m>1$, $n\geq 1$. One example -- a junction with two incoming links and two outgoing links -- is provided in \ref{secAppsj}.
\end{itemize}

This paper is the first to rigorously analyze the continuum junction model that employs a traffic signal control mechanism, and to provide foundation and guidance for the applications of such model, which is an efficient and flexible alternative to the on-and-off signal model. Results developed in this paper have a positive impact on dynamic traffic assignment, especially on the network performance submodel, which describes flow propagation, flow conservation, and travel delay on signalized networks. In particular, when certain types of DTA problems are to be solved using the continuum signal model, our findings made in this paper provide practitioners with suggestions regarding the choice of fundamental diagrams, depending on whether or not spillback occurs, and with ways of assessing the approximation efficacy of the continuum model, based on the error estimates that we established. Immediate applications of the continuum signal model to DTA are under way. It also remains an important aspect of theoretical investigation to extend our methodological framework to other  traffic flow dynamics, such as the link delay model \citep{Friesz1993} and the Vickrey model \citep{Vickrey, GVM1, GVM2}, and to accommodate more complicated turning movements at junctions.


 \appendix

 \section{Two mixed integer linear programs for optimal signal control} \label{secappdx}
 
\subsection{The link-based kinematic wave model (LKWM)}
Discussion of the LKWM below follows \cite{MIPsignal}, and the resulting discrete-time model is  equivalent to the link transmission model proposed by \cite{LTM}. 

Let us consider a homogeneous link $[a,\,b]$, whose dynamic is governed by the LWR model. A triangular fundamental diagram is used with the same set of notations as given in \eqref{triangularfd}. Define a binary variable $\bar r(t)$ that indicates whether the entrance of the link is in the free-flow phase ($\bar r(t)=0$) or in the congested phase ($\bar r(t)=1$). A similar notation $\hat r(t)$ is used for the exit of the link. We also define the entering flow $\bar q(t)$ and the exiting flow $\hat q(t)$ of the link. The variational theory then asserts that 
\begin{align}
\label{latent3}
\bar r(t)&=\begin{cases}
 1,\quad &\hbox{if}\quad \int_0^t \bar q(\tau)\,d\tau ~=~ \int_0^{t-{L\over w}}\hat q(\tau)\,d\tau +\rho_{j}L\\
0, \quad  &\hbox{if}  \quad  \int_0^t \bar q(\tau)\,d\tau ~<~ \int_0^{t-{L\over w}}\hat q(\tau)\,d\tau +\rho_{j}L
\end{cases}
\\
\label{latent4}
\hat r(t)&=\begin{cases}
 0,\quad &\hbox{if}\quad \int_0^{t-{L\over v}}\bar q(\tau)\,d\tau ~=~ \int_0^t\hat q(\tau)\,d\tau\\
1, \quad  &\hbox{if}\quad \int_0^{t-{L\over v}}\bar q(\tau)\,d\tau ~>~ \int_0^t\hat q(\tau)\,d\tau
\end{cases}
\end{align}
where $v$ and $w$ denote the forward and backward wave speeds respectively, $\rho_j$ denotes the jam density and $L$ denotes the link length.

\subsection{Discrete-time formulation of the traffic dynamics}
We discretize Eqn. \eqref{latent3} and \eqref{latent4} to get the mixed integer program. Let us begin with some discrete-time notations for each link $I_i$, where the superscript $k$ always indicates the time step.
\begin{align*}
\bar q_i^k \qquad\qquad &\hbox{the flow at which vehicles enter link }I_i, \\
\hat q_i^k \qquad\qquad &\hbox{the flow at which vehicles exit link }I_i, \\
\bar r_{i}^k \qquad  \qquad & \hbox{the binary variable that indicates the traffic phase at the entrance of } I_i,\\
\hat r_{i}^k \qquad  \qquad & \hbox{the binary variable that indicates the traffic phase at the exit of } I_i\\
S_i^k \qquad\qquad &\hbox{the supply of link }I_i, \\
D_i^k \qquad\qquad &\hbox{the demand of link }I_i,\\
u_i^k\qquad\qquad &\hbox{the binary signal control variable for link } I_i, \\
\eta_i^k\qquad\qquad &\hbox{the continuum priority parameter for link } I_i,
\end{align*}

Fix a time step size $h$, we define $\Delta^f_i\doteq {L_i\over v_ih}$, $\Delta^b_i\doteq {L_i\over w_ih}$. Note that both $\Delta^f_i$ and $\Delta^b_i$ are rounded up to the nearest integer if they are not already integers. We are now ready to state the discrete versions of \eqref{latent3}-\eqref{latent4} as follows.
\begin{align}
\label{panding1}
&\begin{cases}
\displaystyle h\sum_{k=1}^{l-\Delta^b_i} \hat q_i^k -h\sum_{k=1}^{l}\bar q_i^k+\rho_{j,i}L_i~=~\mathcal{M}\,(1-\bar r_i^l)
\\
\displaystyle h\sum_{k=1}^{l-\Delta^b_i} \hat q_i^k -h\sum_{k=1}^{l}\bar q_i^k+\rho_{j,i}L_i~>~-\mathcal{M}\,\bar r_i^l
\end{cases}\qquad\forall i,~~\forall l
\\
\label{panding2}
&\begin{cases}
\displaystyle h\sum_{k=1}^{l-\Delta^f_i} \bar q_i^k -h\sum_{k=1}^{l}\hat q_i^k~=~\mathcal{M}\,\hat r_i^l\\
\displaystyle h\sum_{k=1}^{l-\Delta^f_i} \bar q_i^k -h\sum_{k=1}^{l}\hat q_i^k~>~\mathcal{M}\,(\hat r_i^l-1)
\end{cases}\qquad \forall i,~~\forall l
\end{align}
where $\mathcal{M}>0$ is a large constant, $\rho_{j, i}$ and $L_i$ denote respectively the jam density and length of link $I_i$.  Throughout this section, we stipulate that the entrance of every link of the network remains in the uncongested phase so that spillback does not occur. There are two reasons for this: 1) the non-spillback situation is reasonable  to maintain in a signal optimization process; 2) the continuum model yields a good approximation of the on-and-off model in the absence of spillback. With this in mind, we must have $\bar r_i^k\equiv 0$ for all $i$ and $k$, thus \eqref{panding2} reduces to
\begin{equation}
\label{panding2'}
\displaystyle h\sum_{k=1}^{l-\Delta^f_i} \bar q_i^k -h\sum_{k=1}^{l}\hat q_i^k~=~0\qquad \forall i,~~\forall l
\end{equation}
Moreover, the demand $D_i^k$, whose continuous-time expression is given by \eqref{demanddef}, is now determined via the following inequalities, where $C_i$ denotes the flow capacity of link $I_i$
\begin{equation}\label{inqdemand}
\begin{cases}
C_i+\mathcal{M}(\hat r_i^k-1)~\leq~D_i^k~\leq~C_i
\\
\bar q_i^{k-\Delta^f_i}-\mathcal{M}\hat r_i^k~\leq~D_i^k~\leq~\bar q_i^{k-\Delta_i^f}+\mathcal{M}\hat r_i^k \qquad\forall i,~~\forall k
\end{cases}
\end{equation}

\subsection{Dynamics at signalized junctions}\label{secAppsj}

We relate our expression of the discrete network dynamics to two specific types of signalized junctions depicted in Figure \ref{figtwontw}. 
\begin{figure}[h!]
\centering
\includegraphics[width=.7\textwidth]{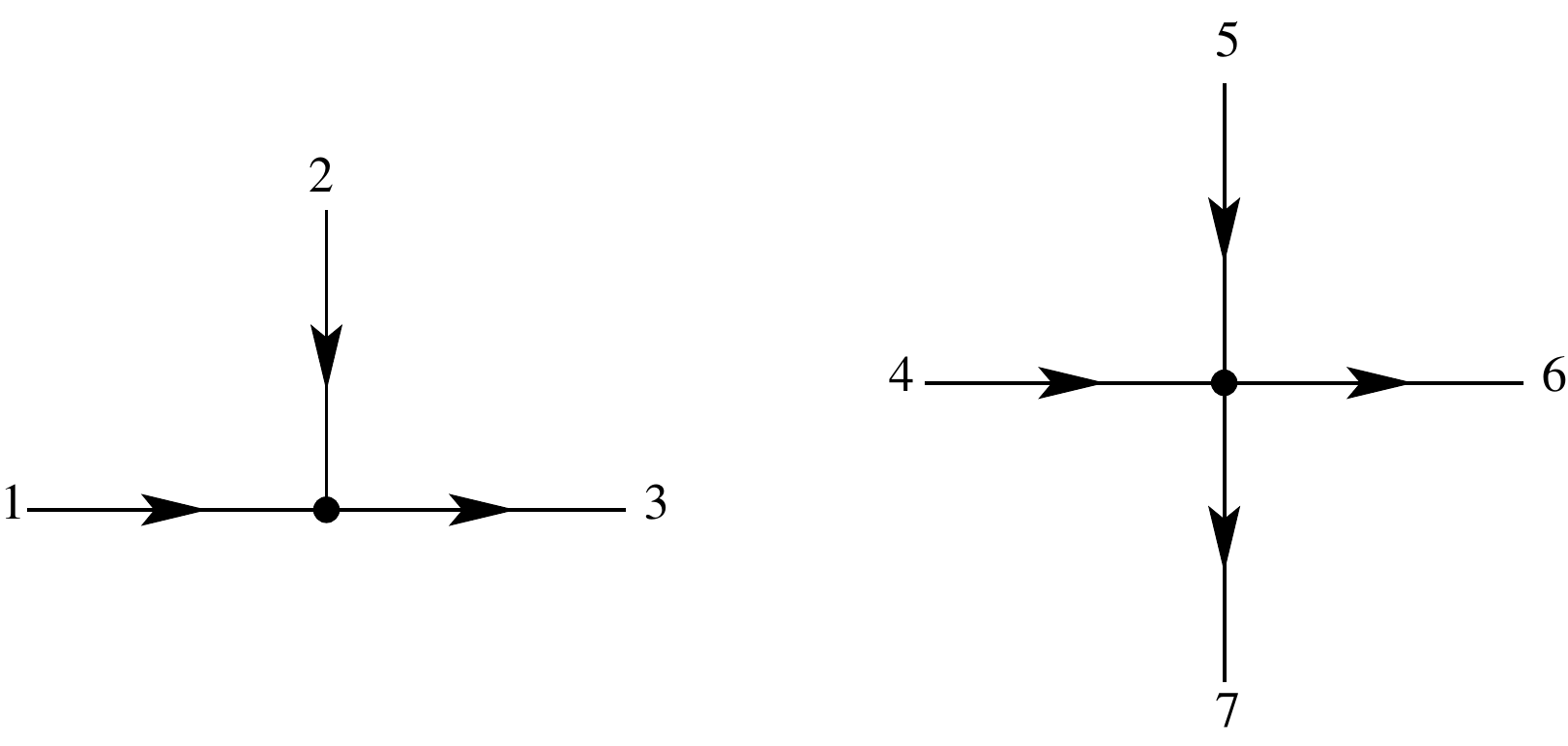}
\caption{Two signalized junctions}
\label{figtwontw}
\end{figure}

\begin{itemize}
\item  Regarding the merge junction on the left of Figure \ref{figtwontw}, the continuous-time dynamics are already given by \eqref{oaodef} and \eqref{continuumdef} for the on-and-off model and the continuum model respectively. Therefore, the discrete-time formulations are:
\begin{align}
\label{oaodefdis3}
\left.
\begin{array}{l}
\hat q_1^k~=~\min\big\{D_1^k,\,\min\{C_1,\,C_3\}\cdot u_1^k\big\}
\\
\hat q_2^k~=~\min\big\{D_2^k,\,\min\{C_2,\,C_3\}\cdot u_2^k\big\}
\\
\bar q_3^k~=~\hat q_1^k+\hat q_2^k
\end{array}
\right\} \qquad\hbox{On-and-off model}
\\ \nonumber
\\
\label{continuumdefdis3}
\left.
\begin{array}{l}
\hat q_1^k~=~\min\big\{D_1^k,\,\min\{C_1,\,C_3\}\cdot \eta_1^k\big\}
\\
\hat q_2^k~=~\min\big\{D_2^k,\,\min\{C_2,\,C_3\}\cdot \eta_2^k\big\}
\\
\bar q_3^k~=~\hat q_1^k+\hat q_2^k
\end{array}
\right\} \qquad\hbox{Continuum model}
\end{align}

\item For the intersection on the right of Figure \ref{figtwontw}, we need to introduce additional turning rates $\alpha_{4,6}+\alpha_{4,7}=1$, $\alpha_{5,6}+\alpha_{5,7}=1$. It is straightforward to verify that the on-and-off and the continuum models are: 
\begin{align}
\label{oaodef4}
\left. 
\begin{array}{l}
\hat q_4^k~=~\min\big\{D_4^k,\,\min\{C_4,\,{C_6\over \alpha_{4,6}},\,{C_7\over \alpha_{4,7}}\}\cdot u_4^k\big\}
\\
\hat q_5^k~=~\min\big\{D_5^k,\,\min\{C_5,\,{C_6\over\alpha_{5,6}},\,{C_7\over \alpha_{5,7}}\}\cdot u_5^k\big\}
\\
\bar q_6^k~=~\alpha_{4,6}\hat q_4^k+\alpha_{5,6}\hat q_5^k,\qquad \bar q_7^k~=~\alpha_{4,7}\hat q_4^k+\alpha_{5,7}\hat q_5^k
\end{array}
\right\} \quad\hbox{On-and-off model}
\\ \nonumber
\\
\label{continuumdef4}
\left. 
\begin{array}{l}
\hat q_4^k~=~\min\big\{D_4^k,\,\min\{C_4,\,{C_6\over \alpha_{4,6}},\,{C_7\over \alpha_{4,7}}\}\cdot \eta_4^k\big\}
\\
\hat q_5^k~=~\min\big\{D_5^k,\,\min\{C_5,\,{C_6\over\alpha_{5,6}},\,{C_7\over \alpha_{5,7}}\}\cdot \eta_5^k\big\}
\\
\bar q_6^k~=~\alpha_{4,6}\hat q_4^k+\alpha_{5,6}\hat q_5^k,\qquad \bar q_7^k~=~\alpha_{4,7}\hat q_4^k+\alpha_{5,7}\hat q_5^k
\end{array}
\right\} \quad\hbox{Continuum model}
\end{align}
\end{itemize}

It remains to express the operator $\min(\cdot)$ appearing in \eqref{oaodefdis3}-\eqref{continuumdef4} as a set of linear inequalities by using additional binary variables, which, due to space limitation, will not be elaborated in this paper. The reader is referred to \cite{MIPsignal} for more detail.

Finally, one has a lot of flexibility in choosing the objective function once the constraints are articulated as above. For our specific example presented in Section \ref{secexperiment} and Figure \ref{figComparisonntw}, the following linear objective function is selected: 
\begin{equation}\label{mipobj}
\hbox{maximize}~~\sum_{k=1}^N {1\over k+1}\left(\hat q_7^k+\hat q_8^k+\hat q_9^k\right)
\end{equation}
where $N$ is the total number of time intervals. Choosing such objective function ensures that the throughput of the network is maximized at any instance of time.

To summarize, for the problem of finding optimal signal timing that avoids spillback, the mixed integer linear program with the on-and-off signal model is given by \eqref{panding1}, \eqref{panding2'}, \eqref{inqdemand}, \eqref{oaodefdis3}, \eqref{oaodef4} and \eqref{mipobj}; the MILP with the continuum model is given by \eqref{panding1}, \eqref{panding2'}, \eqref{inqdemand}, \eqref{continuumdefdis3}, \eqref{continuumdef4} and  \eqref{mipobj}. Notice that both programs may be subject to some additional constraints, e.g., no conflict in signal lights, upper and lower bounds on green and red time, and so forth. These are quite straightforward and are omitted from this paper.

\bibliographystyle{model2-names}
\bibliography{<your-bib-database>}



\end{document}